\documentclass[11pt]{article}

\usepackage{amsmath,amssymb,amscd,amsthm,esint}

\usepackage{graphics,amsmath,amssymb,amsthm,mathrsfs}
\usepackage{mathtools}

\allowdisplaybreaks
\usepackage{amsfonts}
\usepackage{esint}
\usepackage{pifont}
\usepackage{bbding}
\usepackage{latexsym}
\usepackage{mathrsfs}
\usepackage{amsmath}
\usepackage{amsthm}
\usepackage{amssymb}
\usepackage{graphicx}
\usepackage{lineno}
\usepackage{amssymb}

\usepackage[left=1in, right=1in, bottom=1.4in]{geometry}

\newtheorem{theorem}{Theorem}[section]

\newtheorem{lemma}{Lemma}[section]
\newtheorem{remark}{Remark}[section]
\numberwithin{equation}{section}

\newcommand{\R}{\mathbb{R}}

\newcommand{\mc}{\mathcal}
\newcommand{\wt}{\widetilde}

\newcommand{\na}{\nabla}

\newcommand{\va}{\varepsilon}
\newcommand{\al}{\alpha}
\newcommand{\pa}{\partial}
\newcommand{\Om}{\Omega}
\newcommand{\om}{\omega}

\newcommand{\ga}{\gamma}

\newcommand{\be}{\beta}

\date{} 
\begin{document}

\title{Convergence rates in  homogenization of higher order parabolic systems}

\author{Weisheng Niu \thanks{Supported partly by the NSF of China (11701002,11301003)
and NSF of Anhui Province (1708085MA02). }~~~ Yao Xu}
\maketitle
\pagestyle{plain}
\begin{abstract}
This paper is concerned with the optimal convergence rate in  homogenization of higher order parabolic systems with bounded measurable, rapidly oscillating periodic coefficients. The sharp $O(\va)$  convergence rate  in the space $L^2(0,T; H^{m-1}(\Om))$ is obtained for both the initial-Dirichlet problem and the initial-Neumann problem. The duality argument inspired by  \cite{suslinaD2013} is used here.
 \end{abstract}



\section{Introduction}
We consider the sharp convergence rate in periodic homogenization of initial value problems
\begin{equation} \label{eq1}
\begin{cases}
 \pa_tu_\va+\mathcal{L}_\varepsilon u_\varepsilon =f  &\text{ in } \Omega\times (0,T),   \vspace{0.1cm} \\
   u_\va=h   &\text{ on }  \Omega \times \{t=0\},
\end{cases}
\end{equation}
with homogeneous Dirichlet and Neumann boundary data. The linear operator $\mc{L}_\va$ is defined as
 \begin{align} \label{cod0}
  \mc{L}_\varepsilon  =  (-1)^{m}\sum_{|\alpha|=|\beta|=m}
D^\alpha \big(A^{\alpha \beta}(x/\va, t/\va^{2m})D^\beta  \big),
\end{align}
where $\alpha, \beta $
are $d$-dimensional multi-indices with components $\alpha_k, \beta_k, k=1,2,...,d$, and
$$ |\alpha|=\sum_{k=1}^d \alpha_k, ~~D^\alpha=D_{x_1}^{\alpha_1} D_{x_2}^{\alpha_2}\cdot\cdot\cdot D_{x_d}^{\alpha_d}.
$$
The coefficients matrix $A(y,s)=(A_{ij}^{\alpha \beta}(y,s)), 1 \leq i, j \leq n,$  is real, bounded measurable with
\begin{align}\label{cod1}
\sum_{|\al|=|\be|=m} A^{\al\be}(y,s)\xi_\al\xi_\be\geq \mu|\xi|^2, ~~|A(y, s)|\leq\frac{1}{\mu} \quad\text{for }\, a.e.\,  (y,s)\in \R^{d+1},
\end{align}
where  $\mu>0$, $\xi=(\xi_\al)_{|\al|=m}$, $\xi_\al=(\xi^1_\al,...,\xi_\al^n)\in \R^n$.
We also assume that $A$ is $1$-periodic in $(y,s),$ that is,
\begin{align}\label{cod2}
A(y+z,  s+t)=A(y,s)\quad \text{for any } (z,t)\in \mathbb{Z}^{d+1} \text{ and }\, a.e.\, (y,s)\in \mathbb{R}^{d+1}.
\end{align}


Let $H_0^m(\Omega)$ be the conventional $\R^n$-valued Sobolev spaces with dual $H^{-m}(\Om)$.  For $ 0<T<\infty,$ it is known that
under the uniform parabolic condition
(\ref{cod1}),
for any  $ f\in L^2(0,T; H^{-m}(\Omega))$ and any $ h\in L^2(\Om)$,
 initial value problem (\ref{eq1}) with homogeneous Dirichlet data admits a unique weak solution $u_\va$ in the sense that $u_\varepsilon\in L^2(0,T; H_0^m(\Omega))\cap L^\infty(0,T; L^2(\Omega))$,
\begin{align*}
&-\int_{\Om_T} u_\va\pa_t\phi dxdt+
  \sum_{|\alpha|=|\beta|=m} \int_{\Omega_T} A_{ij}^{\alpha \beta}(x/\varepsilon,t/\va^{2m})D^\beta u_{\va j} D^\alpha \phi_i \, dxdt\\
  &  =\int_0^T \langle f(t), \phi(t)\rangle_{H^{-m}(\Om)\times H_0^m(\Om)} dt+\int_\Om h \phi(0) dx
\end{align*}
for any $\phi\in C_c^\infty (\Om\times [0, T))$.

 As we shall prove in Section 2, under periodicity condition (\ref{cod2})
the homogenized problem of (\ref{eq1}) with homogeneous Dirichlet boundary data is given by   \begin{equation}  \label{hoeq1}
 \begin{cases}
 \pa_tu_0+\mathcal{L}_0u_0 =f  &\text{ in } \Omega\times (0,T),   \\
 Tr (D^\gamma u_0)=0 & \text{ on } \partial\Omega \times(0,T),\, 0\leq|\gamma|\leq m-1,  \\
 u_0=h & \text{ on }  \Omega \times \{t=0\},
\end{cases}
\end{equation}
where
\begin{align}\label{hoop}
 \begin{split}
 &\mathcal{L}_0 = (-1)^m \sum_{|\alpha|=|\beta|=m} D^\alpha (\bar{A}^{\alpha \beta}D^\beta ),\\
&\bar{A}_{ij}^{\alpha \beta}=\sum_{|\gamma|=m} \int_Y \Big[A_{ij}^{\alpha \beta}(y,s)
 + A_{i\ell}^{\alpha \gamma}(y,s)D^\gamma \chi_{\ell j}^\beta (y,s)\Big]dyds,
 \end{split}
\end{align}
 with $Y=[-1/2, 1/2)^{d+1}$ and
$\chi= (\chi^\gamma_{ij})$ being the matrix of correctors for the operator $\partial_t+\mathcal{L}_\varepsilon$. 
Moreover, as $\va$ tends to zero $u_\va$ converges strongly to $u_0$ in $L^2(0,T; H^{m-1}_0(\Om))$.

Our first objective is to derive the optimal convergence rate of $u_\va$ to $u_0$.
\begin{theorem}\label{tcd}
Let $\Om$ be a bounded $C^{m,1}$ domain in $\R^d$ and  $0<T<\infty$.  Assume that $A$ satisfies (\ref{cod1}) and (\ref{cod2}),  $f\!\in\! L^2(0,T; H^{-m+1}(\Om)), h\!\in\! L^2(\Om).$ Let $u_\varepsilon, u_0\in L^2(0,T; H_0^m(\Om))$ be the weak solutions to initial-Dirichlet problems (\ref{eq1}) and (\ref{hoeq1}), respectively. If  in addition  $u_0\in L^2(0,T; H^{m+1}(\Om))$, then
\begin{align}\label{tcdre1}
\|u_\va-u_0\|_{L^2(0,T; H_0^{m-1}(\Om))} \leq C \va \Big\{ \|u_0\|_{L^2(0,T; H^{m+1}(\Om))} + \|f\|_{L^2(0,T; H^{-m+1}(\Om))}+\|h\|_{L^2(\Om)} \Big\},
 \end{align}
 where $C$ depends only on $d,n,m,\mu, T $ and $\Om.$
\end{theorem}

The second objective of this paper is to obtain the sharp convergence rate in the homogenization of (\ref{eq1}) with homogeneous Neumann boundary data. Let $H^m(\Omega) $ be the conventional $\R^n$-valued Sobolev spaces with dual $\wt{H}^{-m}(\Omega)$ and let $0<T<\infty$. Write out the complete variational form of (\ref{eq1}) with boundary term formally. Assume that it possesses the following form
\begin{align}\label{neu1}
& - \int_{\Om_T} u_\va\pa_t\phi dxdt+
  \sum_{|\alpha|=|\beta|=m} \int_{\Omega_T} A_{ij}^{\alpha \beta}(x/\varepsilon, t/\va^{2m})D^\beta u_{\va j}(x,t) D^\alpha \phi_i(x,t)\, dxdt \nonumber\\
  &+\sum_{j=1}^{ m-1} \int_0^T \langle N_{m-1-j}u_\va(t), \pa_\nu^j \phi (t) \rangle_{\partial\Om} dt=\int_0^T \langle f(t), \phi(t)\rangle_{\wt{H}^{-m}(\Om)\times H^m(\Om)} dt+\int_\Om h \phi(0) dx
\end{align} for any  $\phi\in C^\infty (\Om\times [0,T])$ with $\phi(T)=0$,
where $\pa_\nu$ is the derivative along the unit outward normal $\nu$ of $\Omega$.  By homogeneous Neumann boundary data, we mean that, formally,
\begin{align}\label{neu2}
N_{m-1-j}u_\va=0\quad \text{ on } \pa\Om\times (0,T), \quad j=0,1,..,m-1.
\end{align}
For $ f\in L^2(0,T; \wt{H}^{-m}(\Omega)), h\in L^2(\Om)$, problem (\ref{eq1}) with homogeneous Neumann boundary data admits a unique weak solution $u_\varepsilon\in L^2(0,T; H^m(\Omega))\cap L^\infty(0,T; L^2(\Omega))$ \cite{dongtran2016}.

Similar to the initial-Dirichlet problem, under the periodicity condition (\ref{cod2})  $u_\va$ converges strongly in $L^2(0,T; H^{m-1}(\Om))$ to the solution $u_0$ of the following initial-Neumann problem
 \begin{equation}  \label{hoeq2}
 \begin{cases}
 \pa_tu_0+\mathcal{L}_0u_0 =f  &\text{ in } \Omega\times (0,T),  \\
 N_{m-1-j} (u_0)=0 & \text{ on } \partial\Omega \times(0,T),    \\
 u_0=h & \text{ on } \Omega \times \{t=0\},
\end{cases}
\end{equation}
  where $\mathcal{L}_0$ is defined as in (\ref{hoop}) (see Theorem \ref{th2.2}).
Parallel to Theorem \ref{tcd}, we have the following sharp convergence rate in homogenization of the initial-Neumann problem.
\begin{theorem}\label{tcn}
Let $\Om$ be a bounded $C^{m,1}$ domain in $\R^d  $ and $0<T<\infty$.  Assume that $A$ satisfies  (\ref{cod1}) and (\ref{cod2}), and $f\!\in\! L^2(0,T; \wt{H}^{-m+1}(\Om)), h\!\in\! L^2(\Om).$ Let $u_\varepsilon, u_0 \in L^2(0,T; H^m(\Om))$ be, respectively, the weak solutions to the initial-Neumann problems (\ref{eq1})  and (\ref{hoeq2}). If in addition $u_0\in L^2(0,T; H^{m+1}(\Om))$, then
\begin{align}\label{tcnre1}
\|u_\va-u_0\|_{L^2(0,T; H^{m-1}(\Om))} \leq C \va \left\{ \|u_0\|_{L^2(0,T; H^{m+1}(\Om))} + \|f\|_{L^2(0,T; \wt{H}^{-m+1}(\Om))}+\|h\|_{L^2(\Om)} \right\},
 \end{align}
 where $C$ depends only on $d,n,m,\mu, T $ and $\Om.$
\end{theorem}

The proof of Theorem 1.1   is mainly based on the duality argument initiated in \cite{suslinaD2013}.
 To adapt the ideas, we first provide the existence results for the matrix of correctors $\chi(y,s)$ and flux correctors $ \mathfrak{B}(y,s)$ (also referred as dual correctors) for operators $\pa_t+\mathcal{L}_\varepsilon (\va>0)$ in Section 2.
 Recall that flux correctors  play an essential role in the investigation on sharp convergence rate in the homogenization of second order elliptic or parabolic systems \cite{jko1994, kls2012,klsa1,shenan2017, gsjfa2017}. In \cite{gsjfa2017}, the flux correctors are obtained by considering a harmonic system with periodic boundary conditions in $\R^{d+1}$ (see Lemma 2.1 therein), which is a modification of the approach for second order elliptic systems. The process  however seems not applicable to higher order parabolic systems. Indeed, following the process we will obtain a degenerate elliptic system in $\R^{d+1}$, which is hard to cope with. Instead, we will modify the construction of flux correctors for elliptic systems in another manner to construct the flux correctors for high order parabolic systems, see Lemmas 2.1 and 2.2. This approach also provides us further regularity results on the flux correctors (see \eqref{le21re2}$_{2}$).  As we can see from the definition of $w_\va$ in \eqref{ome}, the higher regularity ($H^{2m-1}$) on $\mathfrak{B}^{\gamma (d+1) \beta}$ (or $\mathfrak{B}^{\gamma (d+1) \beta}$) are essential, which however is trivial for second order parabolic systems ($m=1$) \cite{gsjfa2017}.

 Since we consider the systems with coarse coefficients, the correctors $\chi(y,s)$  and the flux correctors $ \mathfrak{B}(y,s)$ may be unbounded. Therefore, similar to \cite{gsjfa2017},  in Section 3 we  introduce the smoothing operator $S_\va$ with respect to the space and time variables $x, t$ and establish proper estimates for the smoothing operator.  However, to deal with the higher order operators,  more general estimates are proved by using an approach quite different from \cite{gsjfa2017}.

With preparations in Sections 2 and 3, in Section 4  we introduce the function
\begin{align}\label{ome}
w_\varepsilon(x,t) &=u_\varepsilon(x,t)-u_0(x,t)-\varepsilon^m\sum_{|\gamma|=m}\chi^\gamma (x/\va,t/\va^{2m}) S^2_\varepsilon(D^\gamma  u_0 )\rho_\varepsilon(x)\varrho_\va(t)
\nonumber\\
&\quad+(-1)^{m+1} \sum_{\substack{|\beta|=|\gamma|=m\\ \zeta+\eta =\gamma\\0\leq|\zeta|\leq m-1}}\varepsilon^{m+|\eta|}D^\zeta\mathfrak{B}^{\gamma (d+1) \beta}(x/\va,t/\va^{2m}) D^\eta[S^2_\varepsilon(D^\gamma  u_0 )\rho_\varepsilon(x)\varrho_\va(t)] \nonumber\\ &\quad+(-1)^m\varepsilon^{2m}\sum_{|\alpha|=|\beta|=m}\mathcal{B}^{\alpha\beta}(x/\va,t/\va^{2m})D^\alpha[S^2_\varepsilon(D^\gamma  u_0 )\rho_\varepsilon(x)\varrho_\va(t)],\end{align}
where $ \rho_\va , \varrho_\va$ are proper cut-off functions, $S^2_\va=S_\va\circ S_\va$, see \eqref{def_w} for the details. Then through some delicate analysis and proper use of preparations aforementioned, we prove the following $O\sqrt{\va}$ estimate in $L^2(0,T; H^{m-1}(\Om))$,   \begin{align}\label{pre}
\|\na^m w_\va\|_{L^2(\Om\times (0,T))}\leq C \va^{1/2} \Big\{\|u_0\|_{L^2(0,T; H^{m+1}(\Om))}+ \|f\|_{L^2(0,T; H^{-m+1}(\Om))} +\|h\|_{L^{2}(\Om)}  \Big\}.
\end{align} The above estimate should be comparable to (3.17) or (3.20) in \cite{gsjfa2017}. Yet, we point out that the auxiliary function $w_\va$ is much more complicated than that for second order parabolic systems constructed in \cite{gsjfa2017}. And compared to the proof of (3.17) in \cite{gsjfa2017}, the proof of (\ref{pre}) needs more delicate analysis.  Whence (\ref{pre}) is obtained, the desired estimate (\ref{tcdre1}) follows directly by the duality argument motivated by \cite{suslinaN2013,  gsjfa2017}.  The proof of Theorem 1.2 is completely parallel, and is sketched in Section 5.

As the end of the introduction, let us provide a brief review on the background of convergence rates in quantitative homogenization, which is currently a quite active area of research.
Sharp convergence rates for second order elliptic equations (systems)  has been studied extensively in various circumstances in the past years.
To name but a few, in \cite{grisoas2004, suslinaD2013,suslinaN2013} the optimal $O(\va)$ convergence rate was obtained for second order elliptic equations with Dirichlet or Neumann boundary conditions in $C^{1,1}$ domains. In \cite{kls2012, shenan2017}, the optimal $O(\va)$ and suboptimal convergence rates (like $O(\va \ln\frac{1}{\va})$) were derived for second order elliptic systems with Dirichlet or Neumann boundary conditions in Lipschitz domains. See also \cite{al87, klscpam2014,gsjde2015, shennote2017, shenzhu2017} and references therein for more related results.

For second order parabolic equations with time-independent coefficients, the sharp convergence rate has also been studied widely, see \cite{zhikovRuss2006, SuslinaFunct2010} for the Cauchy problems on the whole space, and \cite{meshkovaFunct2015, meshkovaAppl2016} for the initial boundary value problems in $C^{1,1}$ cylinders. Quantitative estimates for parabolic equations with time dependent coefficients are a bit more intricate and little progress was made until very recently \cite{geng2015uniform, gsjfa2017, Byundcds2017,xu2017quantitative, armstrong2017quantitative}.  In \cite{gsjfa2017} the optimal $O(\va)$ convergence rate in $L^2(0,T; L^2(\Om))$ was obtained in homogenization of second order parabolic systems in $C^{1,1}$ cylinders, while in \cite{xu2017quantitative} the suboptimal $O(\va\ln(1/\va))$ convergence rate in $L^2(0,T; L^2(\Om))$ was obtained for parabolic systems of elasticity in Lipschitz cylinders.  More recently, in \cite{armstrong2017quantitative} the convergence rate and uniform regularity estimates in homogenization of second order stochastic parabolic equations were deeply studied. See also \cite{geng2015uniform, Byundcds2017} for more results on the uniform regularity estimates in the periodic setting.

Homogenization of higher order elliptic equations arises in the study of linear elasticity \cite{lions1978,jko1994,pastukhova2017}, for which the qualitative results have been obtained  for many years \cite{lions1978,jko1994}. Few quantitative results were known in the homogenization of higher order elliptic or parabolic equations until very recently. In \cite{pastukhova2016, pastukhova2017, ks}, the optimal $O(\va)$ convergence rate was established in homogenization of higher-order elliptic equations in the whole space.
In \cite{Suslina2017-D, Suslina2017-N},
some $O(\va)$ two-parameter resolvent estimates were obtained
for more general higher order  elliptic systems with homogeneous Dirichlet or Neumann boundary data in bounded $C^{2m}$ domains.
 Shortly, the sharp convergence rate  and uniform regularity estimates in the homogenization of higher order elliptic systems with symmetric or nonsymmetric coefficients were further studied in \cite{nsx, nxboundary}, see also \cite{xuniusiam2017} for the results in the almost-periodic setting.

 As far as we know, quantitative estimates in homogenization of higher order parabolic equations have not been studied, especially for those with time dependent coefficients. The present paper seems to be the first attempt in this direction. Our results in Theorems \ref{tcd} and \ref{tcn} extend the convergence results for higher order elliptic equations in \cite{pastukhova2016, pastukhova2017,Suslina2017-D, Suslina2017-N} to parabolic systems on the one hand, and on the other hand  they extend the results for second order parabolic systems in \cite{gsjfa2017} to higher order parabolic systems.


\section{Qualitative homogenization}

\subsection{Correctors and flux correctors}
Set $Y=[-\frac{1}{2}, \frac{1}{2})^{d+1}$.
For $1\leq i, j\leq n$ and $d$-dimensional multi-index $\gamma$ of degree $m$, i.e. $|\gamma|=m$, we introduce the matrix of correctors $\chi=\big(\chi_j^\ga(y,s)\big)= \big(\chi^\gamma_{ij}(y,s)\big)$ for the family of operators $\pa_s+\mathcal{L}_\varepsilon(\va>0)$, given by the following cell problem in $Y$,
\begin{align}\label{corrector}
\begin{split}
 &\frac{\pa(\chi_{ij}^\ga)}{\pa s} +(-1)^m\sum_{|\alpha|=|\beta|=m}  D^\alpha \left\{A_{ik}^{\alpha \beta}(y,s)D^\beta \chi_{kj}^\ga\right\}
 = (-1)^{m+1}\sum_{|\alpha|=m}  D^\alpha   A_{ij}^{\alpha \gamma}(y,s)
 ~~    \text{ in } Y,\\
 &\chi_{j}^\gamma (y,s)\quad\text{is 1-periodic in } (y,s),  \quad\text{and}\quad
  \int_Y \chi_{j}^\gamma (y,s)\, dy\,ds=0, \quad 1\leq j\leq n,
 \end{split}
\end{align}
where $ \chi_j^\gamma =\big( \chi_{1j}^\gamma, \chi_{2j}^\gamma,...., \chi_{nj}^\gamma \big)$ for each fixed $\gamma$ and $j$. Under conditions (\ref{cod1}) and (\ref{cod2}),  the existence of $\chi_j^\gamma$ follows from standard existence results of general parabolic systems.

For $1\leq i,j\leq n$ and $d$-dimensional multi-indices $\alpha,\beta$ of degree $m$, we set
\begin{align}\label{abar}
\bar{A}_{ij}^{\alpha \beta}&= \int_Y  \Big\{A_{ij}^{\alpha \beta}(y,s)
+ \sum_{|\gamma|=m}A_{i\ell}^{\alpha \gamma}(y,s)D^\gamma \chi_{\ell j}^\beta (y,s) \Big\}\,dyds\nonumber\\
&=\sum_{|\eta|=|\zeta|=m}\int_Y A^{\zeta\eta}(y,s)D^\eta\Big(\frac{1}{\be!}y^\be e_j+\chi_j^\be(y,s)\Big) \cdot D^\zeta \Big( \frac{1}{\al!}y^\al e_i\Big)\,dyds,
\end{align}
where $y^\al=y_1^{\al_1}y_2^{\al_2},...,y_d^{\al_d}$.

In this section, the symbols $\imath, \jmath$ may equal $(d+1)$ or represent $d$-dimensional multi-indices of degree $m$.
We define $B_{ij}^{\imath\beta}$ for $|\beta|=m$ by
\begin{align}\label{duc}
B_{ij}^{\imath\beta} =\left\{
  \begin{array}{ll}
   A_{ij}^{\alpha\beta} + \sum_{|\gamma|=m}A_{ik}^{\alpha\gamma}  (D^\gamma \chi_{kj}^\beta) -\bar{A}_{ij}^{\alpha\beta}, &\textrm{ if } \imath=\alpha,\\
   (-1)^m \chi_{ij}^\be, &\textrm{ if } \imath=d+1,
  \end{array}
\right.
\end{align}
where $\alpha$ is a $d$-dimensional multi-index of degree $m$.
Also, we define
\begin{align*}
  \widehat{u}(s)=\int_{Y_d}u(y, s)dy,
\end{align*}
for 1-periodic function $u(y, s)$ in $\R^{d+1}$, where $Y_d=[-1/2, 1/2)^d$.

The following lemma gives the existence of the matrix of flux correctors for the family of operators $ \pa_s+ \mathcal{L}_\varepsilon (\va>0)$.
\begin{lemma}\label{le2.1}
For any $1\leq i,j\leq n$ and $d$-dimensional multi-index $\beta$ with
$|\beta|=m$, there exist 1-periodic functions $\mathfrak{B}_{ij}^{\imath\jmath\beta}(y, s)$ in $\R^{d+1}$ such that
\begin{align}\label{le21re1}
\mathfrak{B}_{ij}^{\imath\jmath\beta} =- \mathfrak{B}_{ij}^{\jmath\imath\beta},\quad \sum_{|\gamma|=m}D^{\ga} \mathfrak{B}_{ij}^{\ga\jmath\beta}(y,s)+ \pa_s \mathfrak{B}_{ij}^{(d+1)\,\jmath\beta}(y,s)= B_{ij}^{\jmath\beta}(y,s)-\widehat{B}_{ij}^{\jmath\beta}(s).
\end{align}
Furthermore, there exists a constant $C$ depending only on $d, n, m, \mu$ such that
\begin{align}\label{le21re2}
  \begin{split}
    &\|\mathfrak{B}_{ij}^{\imath\jmath\be} \|_{L^2(-1/2, 1/2; H^m(Y_d))}\leq  C \quad\textrm{if } \imath, \jmath \textrm{ are $d$-dimensional multi-indices of degree $m$},\\
    &\|\mathfrak{B}_{ij}^{\imath\jmath\be} \|_{L^2(-1/2, 1/2; H^{2m}(Y_d))}\leq  C \quad\text{if }  \imath \textrm{ or } \jmath =d+1.
    \end{split}
\end{align}
\end{lemma}
\begin{proof} For simplicity of presentations, let us suppress the subscripts $i,j$. Since for $|\beta|=m$, $B^{\imath\be}(y,s)$ are $1$-periodic in $\R^{d+1}$, and for any $s\in \R$, $$\int_{Y_d}\big[B^{\imath\beta}(y,s)-\widehat{B}^{\imath\beta}(s)\big]dy=0,$$
there exist $f^{\imath\beta}(\cdot, s)\in H^{2m}(Y_d)$ such that
\begin{equation*}
\begin{cases}
\Delta_d^mf^{\imath\beta}(y, s)=B^{\imath\beta}(y,s)-\widehat{B}^{\imath\beta}(s)\quad \textrm{in}\quad \R^d,\\
f^{\imath\beta}(\cdot, s)\ \textrm{is 1-periodic}\quad \textrm{in}\quad \R^d,
\end{cases}
\end{equation*}
where $\Delta_d$ denotes the Laplacian in $\R^d$.
We define for $|\beta|=m$,
\begin{align*}
  \mathfrak{B}^{\imath\jmath\beta}=\left\{
  \begin{array}{ll}
    D^\gamma f^{\alpha\beta}-D^\alpha f^{\gamma\beta},&\textrm{if }~ \imath=\gamma, \jmath=\alpha,\\
    D^\alpha f^{(d+1)\,\beta},&\textrm{if }~ \imath=\alpha, \jmath=d+1,\\
    -D^\alpha f^{(d+1)\,\beta},&\textrm{if }~\imath=d+1, \jmath=\alpha,\\
    0,&\textrm{if }~ \imath=\jmath=d+1,
  \end{array}\right.
\end{align*}
where $\alpha, \gamma$ are $d$-dimensional multi-indices of degree $m$. Obviously, $\mathfrak{B}^{\imath\jmath\beta} =- \mathfrak{B}^{\jmath\imath\beta}$. Since
\begin{align}\sum_{|\al|=m} D^{\al} B^{\al \beta} (y,s) + \pa_s B^{(d+1)\,\be}(y,s) =0,\label{le21eq}\end{align} we have  $$\Delta_d^m\Big[\sum_{|\al|=m} D^{\al} f^{\al \beta} (y,s) + \pa_s f^{(d+1)\,\be}(y,s)\Big]=0.$$
By the Liouville property for $\Delta_d^m$ and the  periodicity of $f$, we know that $\sum_{|\al|=m} D^{\al} f^{\al \beta} (\cdot, s) + \pa_s f^{(d+1)\,\be}(\cdot, s)$ is a constant. Consequently, it's not hard to verify that
$$\sum_{|\gamma|=m}D^{\ga} \mathfrak{B}^{\ga\jmath\beta}(y,s)+ \pa_s \mathfrak{B}^{(d+1)\,\jmath\beta}(y,s)= B^{\jmath\beta}(y,s)-\widehat{B}^{\jmath\beta}(s).$$

Moreover, note that for any $|\alpha|=|\beta|=|\gamma|=m$, $$\|\mathfrak{B}^{\gamma\alpha\beta}(s)\|_{H^m(Y_d)}\leq C\big(\|f^{\alpha\beta}(s)\|_{H^{2m}(Y_d)}+\|f^{\gamma\beta}(s)\|_{H^{2m}(Y_d)}\big)\leq C\big(\|B^{\alpha\beta}(s)\|_{L^2(Y_d)}+\|B^{\gamma\beta}(s)\|_{L^2(Y_d)}\big),$$which implies the first estimate in \eqref{le21re2}. Similarly, the second part  of \eqref{le21re2} follows from $$\|\mathfrak{B}^{(d+1)\,\jmath\beta}(s)\|_{H^{2m}(Y_d)}\leq C\|f^{(d+1)\,\beta}(s)\|_{H^{3m}(Y_d)}\leq C\|B^{(d+1)\,\beta}(s)\|_{H^m(Y_d)}=C\|\chi^\beta(s)\|_{H^m(Y_d)}.$$
The proof is complete.
 \end{proof}

 \begin{lemma}\label{lem22} Let $B^{\imath\beta}$ be defined as in \eqref{duc}. Then $\widehat{B}^{(d+1)\,\beta}(s)\equiv0$, and moreover
   for $|\alpha|=|\beta|=m$, there exist 1-periodic functions $\mathcal{B}^{\alpha\beta}(s)$ in $\R$ such that $\partial_s\mathcal{B}^{\alpha\beta}(s)=\widehat{B}^{\alpha\beta}(s)$ and $$\|\mathcal{B}^{\alpha\beta}\|_{H^1([-1/2, 1/2])}\leq C\|\widehat{B}^{\alpha\beta}\|_{L^2([-1/2, 1/2])}$$ for some positive constant $C$.
\end{lemma}
\begin{proof}
Integrating equation \eqref{le21eq} over $Y_d$, we get $\partial_s\widehat{B}^{(d+1)\, \beta}(s)=0$, which together with the fact $\int_{-1/2}^{1/2}\widehat{B}^{(d+1)\, \beta}(s)ds=0$ implies that $\widehat{B}^{(d+1)\,\beta}(s)\equiv0$. Moreover, note that $\int_{-1/2}^{1/2}\widehat{B}^{\alpha\beta}(s)ds=0$. By setting $\mathcal{B}^{\alpha\beta}(s)=\int_0^s\widehat{B}^{\al\beta}(s)ds$, we get the desired function. The proof is complete.
 \end{proof}

Let $\mathcal{L}^*_\varepsilon$ be the adjoint operators of $\mathcal{L}_\varepsilon$,  i.e.,
\begin{align}\label{adj}
\mathcal{L}^*_\varepsilon=  (-1)^{m}\sum_{|\alpha|=|\beta|=m} D^\alpha\big(A^{*\alpha\beta}
(x/\va, t/\va^{2m}) D^\beta  \big),
\end{align} where $A^*=(A_{ij}^{*\alpha\beta})= ( A_{ji}^{\beta \alpha}).$
Parallel to (\ref{corrector}), we can introduce the matrix of correctors $\chi^*= (\chi^{*\al}_{i})= (\chi^{*\al}_{k i})$ for the operator    $-\partial_s+\mathcal{L}^*_\varepsilon$, where $(\chi^{*\al}_{i})= \big( \chi_{1i}^{*\al} (y), \chi_{2i}^{*\al} (y),...., \chi_{ni}^{*\al} (y)\big)$ is the solution to the following cell problem,
 \begin{align}\label{dualcorrector}
\begin{split}
 &-\frac{\pa(\chi_{ki}^{*\al})} {\pa s}  +(-1)^m\sum_{|\eta|=|\beta|=m}  D^\eta \Big\{A_{k\ell}^{*\eta \beta}(y,s)D^\beta \chi_{\ell i}^{*\al}\Big\}
 = (-1)^{m+1}\sum_{|\eta|=m}  D^\eta   A_{ki}^{*\eta \al}(y,s)
 ~~    \text{ in } Y,\\
 &\chi_i^{*\al} (y,s)\quad\text{is 1-periodic in } (y,s), \quad\text{and}\quad
  \int_Y \chi_i^{*\al} (y,s)\, dy\,ds=0, \quad 1\leq i\leq n.
 \end{split}
\end{align}
We can also introduce   $ \mathfrak{B}^{*\imath \jmath \beta}(y,s)$ and $ \mathcal{B}^{*\alpha\beta}(s)$ as Lemmas \ref{le2.1} and \ref{lem22}.
It is not difficult to see that $ \chi^{*\ga},  \mathfrak{B}^{*\imath \jmath \beta}$ and  $\mathcal{B}^{*\alpha\beta}(s) $ satisfy the same properties as $\chi^\ga, \mathfrak{B}^{\imath\jmath\beta}$ and $\mathcal{B}^{\alpha\beta}(s) $ respectively, since $A^*$ satisfies the same conditions as $A$.

Taking $\chi^{*\al}_i$ and $\chi_j^\ga$ as  test functions in (\ref{corrector}) and (\ref{dualcorrector}) respectively, we get
\begin{align}
&\int_{-1/2}^{1/2}\big\langle\frac{\pa(\chi_{j}^\ga)}{\pa s}, \chi_i^{*\al}\big\rangle\, ds
 = -\!\sum_{|\eta|=|\be|=m} \int_Y \Big\{ A^{\eta \beta}  D^\beta \chi_{j}^\ga D^\eta \chi_i^{*\al}+A^{\eta \be} D^\be\Big(\frac{y^\ga}{\ga!}e_j\Big)D^\eta \chi_i^{*\al}\Big\}\,dyds,\nonumber\\
 & \int_{-1/2}^{1/2}\big\langle\frac{\pa(\chi_{i}^{*\al})}{\pa s}, \chi_j^{\ga}\big\rangle\, ds =\sum_{|\eta|=|\beta|=m}\int_Y  \Big\{A^{*\eta \beta} D^\beta \chi_{i}^{*\al} D^\eta \chi_j^{\ga}+ A^{*\eta \be} D^\be\Big(\frac{y^\al}{\al!}e_i\Big)D^\eta \chi_j^{\ga} \Big\}\, dyds\nonumber
\end{align}
which by summation implies that
\begin{align*}
&\sum_{|\eta|=|\be|=m} \int_Y A^{\be\eta}(y,s)D^\eta \chi_j^{\ga} (y,s)D^\be(\frac{y^\al}{\al!}e_i)\, dyds\nonumber\\
& =\sum_{|\eta|=|\be|=m} \int_Y A^{*\eta \be}(y,s) D^\be(\frac{y^\al}{\al!}e_i)D^\eta \chi_j^{\ga}(y,s)\, dyds\nonumber\\
& =\sum_{|\eta|=|\be|=m} \int_Y A^{\eta \be}(y,s) D^\be\Big(\frac{y^\ga}{\ga!}e_j\Big)D^\eta \chi_i^{*\al}(y,s)\, dyds\nonumber\\
& = \sum_{|\eta|=|\be|=m} \int_Y A^{*\be\eta}(y,s) D^\eta \chi_i^{*\al}(y,s) D^\be\Big(\frac{y^\ga}{\ga!}e_j\Big)\, dyds.
\end{align*}
In view of (\ref{abar}), this provides another expression of $\bar{A}$ in terms of $A^*$ and $\chi^{*}$,
\begin{align}\label{abar1}
\bar{A}_{ij}^{\al\ga}&=\sum_{|\eta|=|\be|=m} \int_Y A^{\be\eta}(y,s)D^\eta \Big(\chi_j^{\ga}(y,s)+ \frac{y^\ga}{\ga!}e_j \Big)D^\be\Big(\frac{y^\al}{\al!}e_i\Big)\, dyds\nonumber\\
&= \sum_{|\eta|=|\be|=m} \int_Y A^{*\be\eta}(y,s) D^\eta \Big(\chi_i^{*\al}(y,s)+ \frac{y^\al}{\al!}e_i   \Big)D^\be\Big(\frac{y^\ga}{\ga!}e_j\Big)\, dyds.
\end{align}

\subsection{Effective operators and homogenized systems}
In this part, we prove that the effective  operator for  $\pa_t+\mathcal{L}_\varepsilon$ is $\pa_t +\mathcal{L}_0 $, where $\mc{L}_0$ is defined as in (\ref{hoop}). Let $ \bar{A}=(\bar{A}_{ij}^{\alpha \beta})$ be defined as in (\ref{abar}). In view of (\ref{corrector}), we have
\begin{align*}
\bar{A}_{ij}^{\alpha \beta}&=\int_Y \sum_{|\zeta|=|\eta|=m}A^{\zeta\eta} D^\eta\Big(\frac{y^\be}{\be!} e_j+\chi_j^\be \Big) \cdot D^\zeta \Big(\frac{y^\al}{\al!} e_i+\chi_i^\al \Big)\,
 dyds+\int_{-1/2}^{1/2}\big\langle\pa_s \chi^\be_j, \chi^\al_i \big\rangle\, ds.
\end{align*}
Therefore,
\begin{align*}
 \bar{A}_{ij}^{\alpha \beta}\xi^i_\al \xi^j_\be &=\sum_{|\zeta|=|\eta|=m}\int_Y A^{\zeta\eta} D^\eta\Big(\frac{y^\be}{\be!} e_j\xi_\be^j+\chi_j^\be \xi_\be^j\Big) \cdot D^\zeta \Big(\frac{y^\al}{\al!} e_i\xi_\al^i+\chi_i^\al \xi_\al^i\Big) \, dyds\\
 &\quad+
\int_{-1/2}^{1/2}\big\langle\pa_s \chi^\be_j \xi_\be^j, \chi^\al_i \xi_\al^i \big\rangle\, ds,
\end{align*}
for any $\xi=(\xi_\al)_{|\al|=m}$ with $ \xi_\al=(\xi_\al^i)\in\R^n.$
Note that
$$
\sum_{|\al|=|\be|=m} \int_{-1/2}^{1/2}\big\langle\pa_s \chi^\be_j \xi_\be^j, \chi^\al_i \xi_\al^i \big\rangle\, ds =
\frac{1}{2}\sum_{|\al|=|\be|=m}\int_{-1/2}^{1/2}\pa_s \big\langle \chi^\be_j \xi^j_\be, \chi^\al_i \xi^i_\al \big\rangle\, ds=0.
$$
Using (\ref{cod1}) and integration by parts we obtain that,
\begin{align}\label{ellipticity}
 \sum_{|\al|=|\be|=m} \bar{A}_{ij}^{\alpha \beta}\xi^i_\al \xi^j_\be &\geq \mu \sum_{|\zeta|=m}\!\int_Y\sum_{|\be|=m}\!D^\zeta\Big\{ \frac{y^\be}{\be!} e_j\xi_\be^j\!+\!\chi_j^\be(y,s)\xi_\be^j \Big\}
  \cdot \sum_{|\al|=m}\!D^\zeta\Big\{ \frac{y^\al}{\al!} e_i\xi_\al^i\!+\!\chi_i^\al(y,s)\xi_\al^i \Big\}\,
 \nonumber\\
 &= \mu \sum_{|\zeta|=m}|\xi_\zeta|^2 + \mu\sum_{|\zeta|=m} \int_Y   D^\zeta\Big\{\sum_{|\be|=m}\chi_j^\be(y,s)\xi_\be^j \Big\}  D^\zeta\Big\{\sum_{|\al|=m}\chi_i^\al(y,s)\xi_\al^i\Big\} \nonumber\\
&\geq \mu \sum_{|\zeta|=m}|\xi_\zeta|^2,
\end{align}
which, combined with \eqref{abar}, implies that $\bar{A}$ satisfies the ellipticity condition (\ref{cod1}) with $1/\mu$ replaced by some constant $\mu_0,$ depending only on $d,n,m$ and $\mu$. Therefore, for any $f\in  L^2(0,T; H^{-m}(\Om)),  h\in L^2(\Om)$, problem (\ref{hoeq1}) admits a unique solution $u_0 \in L^2(0,T; H_0^m(\Omega)) \cap L^\infty(0,T; L^2(\Omega)).$

\begin{theorem}\label{th2.1}
Let $\Om$ be a bounded Lipschitz domain in $\R^d.$ Assume that $A$ satisfies conditions (\ref{cod1})--(\ref{cod2}), and $f\in  L^2(0,T; H^{-m}(\Om)), h\in L^2(\Om).$  Let $u_\va, u_0\in L^2(0,T; H^m_0(\Om))\cap L^\infty(0,T; L^2(\Om))$ be the unique weak solutions to initial-Dirichlet problems (\ref{eq1}) and (\ref{hoeq1}), respectively. Then as $\va\longrightarrow 0$,
\begin{align}\label{t21re1}
u_\va \longrightarrow u_0  \text{ weakly in } L^2(0,T; H^m_0(\Om)) \text{ and strongly in }  L^2(0,T; H^{m-1}(\Om)).
\end{align}
\end{theorem}
\begin{proof}
The proof is adapted from Theorem 2.1 in \cite[p.140]{lions1978}, where similar results was proved for second order parabolic equations.
Note that $u_\va, \pa_t u_\va$ are uniformly (in $\va$) bounded in $L^2(0,T; H^m_0(\Om))$ and $L^2(0,T; H^{-m}(\Om))$, respectively, $\sum_{|\be|=m}A_\va^{\al\be} D^\be u_\va $ is uniformly bounded in $ L^2(0,T; L^2(\Om))$ for all $|\al|=m$, where $ A_\va^{\al\be}(x,t)=A^{\al\be}(x/\va, t/\va^{2m}).$
Up to subsequences, we may assume that there exists a function $u_0$ such that \begin{align} \label{pt210}
\begin{split}
&u_\va \longrightarrow u_0 \text{ weakly in } L^2(0,T; H^m_0(\Om)), \\
 &\pa_t u_\va \longrightarrow \pa_t u_0 \text{ weakly in } L^2(0,T; H^{-m}(\Om)),  \\
 & \sum_{|\be|=m}A_\va^{\al\be} D^\be u_\va \longrightarrow   \Psi^{\al}(x, t) \text{ weakly in } L^2(0,T; L^2(\Om)),\\
 &u_\va \longrightarrow u_0 \text{ strongly in } L^2(0,T; H^{m-1}(\Om)),
 \end{split}
\end{align}
where the last convergence result in (\ref{pt210}) follows from the well-known Aubin-Simon type compactness result.
Moreover,
\begin{align*}
\pa_t u_0 +(-1)^m \sum_{|\al|=m} D^\al \Psi^\al=f  \quad\quad \text{in } L^2(0,T; H^{-m}(\Om)).
\end{align*}
Taking $\phi\in C^1([0,T]; H^m_0(\Om))$ with $\phi(T)=0$ as a test function, we obtain that
\begin{align} \label{pt2101}
   -\int_\Om u_0(0) \phi(0)\, dx - \int_{\Om_T}  u_0  \pa_t \phi  \,dxdt+\sum_{|\al|=m}  \int_{\Om_T} \Psi^\al  D^\al \phi  \,dxdt  =\int_0^T \big\langle f(t), \phi(t)\big\rangle  \,dt,
\end{align}
where $\Om_T=\Om\times (0,T)$. Hereafter, let us denote the product of $H^{-m}(\Om)$ and $H_0^m(\Om)$ as $\langle, \rangle$  for short. On the other hand, taking such a $\phi$ as a test function in (\ref{eq1}) and passing to the limits, we obtain that
 \begin{align*}
 -\int_\Om h \phi(0)\, dx -  \int_{\Om_T}  u_0  \pa_t \phi \, dxdt +\sum_{|\al|=m} \int_{\Om_T} \Psi^\al  D^\al \phi \, dxdt   =\int_0^T \big\langle f(t), \phi(t)\big\rangle  \,dt,
 \end{align*}
which, combined with (\ref{pt2101}), implies that $u_0(0)=h.$
Therefore, to verify that $u_0$ is a weak solution of (\ref{hoeq1}), it remains to prove
\begin{align}\label{pt2102}
\Psi^\al= \sum_{|\be|=m}\bar{A}^{\al\be} D^\be u_0.
\end{align}

 For positive integer $k$, let
$$\mathfrak{P}_k=\Big\{(P^1_k, P^2_k,..., P^n_k)\mid  P_k^i \text{ are homogeneous polynomials of $y$ of degree } k  \Big\}.
$$ For $P_m\in \mathfrak{P}_m$, let $\om$ be the weak solution to the cell problem
\begin{align}\label{pt211}
\begin{split}
 & -\pa_s\om  +(-1)^m\sum_{|\alpha|=|\beta|=m}  D^\alpha \big(A^{*\alpha \beta} D^\beta \om\big)
 = (-1)^{m+1}\sum_{|\alpha|=m}  D^\alpha  \big(A^{*\alpha \be} D^\be P_m\big)
\quad \text{in } Y,\\
 &\,\,\om (y,s) \text{ is 1-periodic in } (y,s)  \quad \text{and} \quad
 \int_Y \om (y,s)\, dy\,ds=0.
 \end{split}
\end{align}
Setting $$\theta_\va=\theta_\va(x,t)=\va^m \left\{\om(x/\va, t/ \va^{2m}) + P_m(x/\va)\right\},$$ it is not difficult to find that, as $\va$ tend to zero,
\begin{align}\label{pt212}
 &\theta_\va  \longrightarrow P_m \text{ strongly  in } L^2(0,T; H^{m-1}(\Om)) \text{ and weakly  in } L^2(0,T; H^{m}(\Om)),
 \end{align}
 \begin{align}\label{pt2122}
&-\pa_t \theta_\va+(-1)^m\sum_{|\alpha|=|\beta|=m}  D^\alpha \big(A_\va^{*\alpha \beta} D^\beta \theta_\va\big)=0 \quad \text{in } \R^{d+1},
\end{align}
where $ A_\va^{*\alpha \beta}(x,t)=A^{*\alpha \beta}(x/\va, t/\va^{2m}).$
For any $\phi\in C_c^\infty(\Om\times(0,T)),$ we deduce from (\ref{eq1}) and (\ref{pt2122}) that
\begin{align}\label{pt213}
\begin{split}
&\int_0^T\big\langle \pa_t u_\va, \phi \theta_\va\big\rangle \, dt+ \sum_{|\al|=|\be|=m} \int_{\Om_T} A_\va^{\al\be} D^\be u_\va D^\al (\phi\theta_\va)\, dxdt=\int_0^T\big\langle f, \phi \theta_\va\big\rangle\, dt,\\
&-\int_0^T\big\langle \pa_t \theta_\va, \phi u_\va\big\rangle  \, dt+\sum_{|\al|=|\be|=m} \int_{\Om_T} A_\va^{*\al\be} D^\be \theta_\va D^\al (\phi u_\va)\, dxdt=0.
\end{split}
\end{align}
 Subtracting the second equality from the first one in (\ref{pt213}), it yields
 \begin{align}\label{pt214}
 &\sum_{\substack{|\al|=|\be|=m\\ \eta+\zeta=\al,   |\zeta|\leq m-1}} C(\zeta) \Big\{ \int_{\Om_T} A_\va^{\al\be} D^\be u_\va D^\eta\phi D^\zeta\theta_\va\, dxdt - \int_{\Om_T} A_\va^{*\al\be} D^\be \theta_\va D^\eta \phi D^\zeta u_\va\, dxdt \Big\}\nonumber\\
 &= \int_0^T\langle u_\va, \,\theta_\va \pa_t\phi\rangle dt +\int_0^T\langle f,\, \phi \theta_\va\rangle\, dt,\quad\quad C(\zeta)=\frac{\al!}{\zeta!(\al-\zeta)!}
 \end{align}
Thanks to the convergence results for $u_\va$ and $\theta_\va$ (see (\ref{pt210}) and (\ref{pt212})),  up to subsequences,
\begin{align}\label{pt215}
\begin{split}
&\int_0^T\big\langle u_\va, \theta_\va \pa_t\phi\big\rangle\, dt +\int_0^T\big\langle f, \phi \theta_\va\big\rangle\, dt \longrightarrow \int_0^T\big\langle u_0, \pa_t\phi P_m\big\rangle dt +\int_0^T\big\langle f, \phi P_m\big\rangle\, dt,\\
& \sum_{|\be|=m } \int_{\Om_T} A_\va^{\al\be} D^\be u_\va D^\eta\phi D^\zeta\theta_\va\, dxdt\longrightarrow  \int_{\Om_T} \Psi^{\al} D^\eta\phi D^\zeta P_m \, dxdt,
\end{split}\end{align}
for  $ \zeta+\eta=\al, |\zeta|\leq m-1.$ On the other hand, taking $\phi P_m$ as a test function in (\ref{eq1}) and passing to the limit in $\va,$ we get
\begin{align}\label{pt216}
-\int_0^T\langle  u_0,  \pa_t \phi P_m\rangle\, dt+\sum_{ |\al|=m}  \int_{\Om_T} \Psi^{\al} D^\al(\phi P_m)\, dxdt
 =  \int_0^T\big\langle f, \phi P_m\big\rangle\, dt.
\end{align}
By combing  (\ref{pt214}), (\ref{pt215})  and (\ref{pt216}), we obtain that
\begin{align}\label{pt217}
\sum_{\substack{|\al|=|\be|=m\\ \eta+\zeta=\al,   |\zeta|\leq m-1}}C(\zeta)  \int_{\Om_T} A_\va^{*\al\be} D^\be \theta_\va D^\eta \phi D^\zeta u_\va\, dxdt \longrightarrow - \sum_{|\al|=m} \int_{\Om_T} \Psi^{\al} D^\al P_m \phi\, dxdt.
\end{align}
Note that $  A^{*\al\be}(y,s) D^\be \theta(y)$ is periodic and uniformly bounded in $ L^2(0,T; L^2(\Om))$. By the well-known weak convergence result of rapidly oscillating periodic functions \cite[p.5]{jko1994},
\begin{align} \label{pt218}
 \sum_{|\be|=m} A_\va^{*\al\be} D^\be \theta_\va \quad\text{converges weakly to }  \sum_{|\be|=m} \int_Y A^{*\al\be}(y,s) \big(D^\be \om(y,s)+D^\be P_m \big)\, \doteq \mathcal{M}^\al
\end{align} in $L^2(0,T; L^2(\Om))$,
which implies that the l.h.s. of (\ref{pt217}) converges to
\begin{align}\label{pt219}
 &\sum_{\substack{\eta+\zeta=\al\\ |\al|=m,  |\zeta|\leq m-1}}C(\zeta)  \int_{\Om_T}\mathcal{M}^\al D^\eta \phi D^\zeta u_0\, dxdt \nonumber\\
&=  \sum_{|\al|=m}  \int_{\Om_T} \mathcal{M}^\al D^\al (\phi u_0)\, dxdt - \sum_{|\al|=m}   \int_{\Om_T} \mathcal{M}^\al D^\al u_0\phi \, dxdt\nonumber\\
 & =- \sum_{|\al|=m} \int_{\Om_T}  \mathcal{M}^\al D^\al u_0\phi \, dxdt.
\end{align}
In view of (\ref{pt217}), we get
\begin{align}
\sum_{|\al|=m} \int_{\Om_T} \Psi^{\al} D^\al P_m \phi\, dxdt= \sum_{|\al|=m} \int_{\Om_T} \mathcal{M}^\al D^\al u_0\phi\, dxdt, \quad\forall \phi\in C_c^\infty(\Om\times(0,T)).
\end{align}
Hence $$\sum_{|\al|=m} \Psi^{\al} D^\al P_m =\sum_{|\al|=m} \mathcal{M}^\al D^\al u_0.$$
For any $d$-dimensional multi-index $\ga$ with $|\ga|=m$, set $P_m=\frac{1}{\ga!}y^\ga e_i.$  Then by (\ref{pt211}) and (\ref{pt218}), we have  $\om=\chi^{*\ga}_i$  and
\begin{align*}
\Psi^\ga_i =\sum_{|\al|=m} \Big\{\sum_{|\zeta|=|\eta|=m} \int_Y A^{*\zeta\eta}(y,s) D^\eta\big(\chi^{*\ga}_i +\frac{y^\ga}{\ga!}e_i \big)  D^\zeta\big(\frac{y^\al}{\al!}e_j \big)\, dyds\Big\} D^\al u_{0j}
 = \sum_{|\al|=m} \bar{A}_{ij}^{\ga\al}D^\al u_{0j},
\end{align*}
which is exactly (\ref{pt2102}).
The proof is thus complete.
\end{proof}

Similar to Theorem \ref{th2.1}, we can prove that the homogenized operator for $-\pa_t + \mathcal{L}^*_\va$ is given by
\begin{align}\label{adj0}
-\pa_t+ \mathcal{L}^*_0  = -\pa_t+(-1)^m \sum_{|\alpha|=|\beta|=m}  D^\alpha (\bar{A}_{ij}^{*\alpha \beta}D^\beta ),
\end{align}
where $ \bar{A}^*=(\bar{A}_{ij}^{*\alpha\beta})= ( \bar{A}_{ji}^{\beta \alpha}).$ Furthermore, the same argument also gives the homogenized system for the initial value problem (\ref{eq1}) with homogeneous Neumann boundary data.
\begin{theorem}\label{th2.2}
Let $\Om$ be a bounded Lipschitz domain in $\R^d.$ Assume that $A$ satisfies conditions (\ref{cod1})--(\ref{cod2}), and $f\in  L^2(0,T; \wt{H}^{-m}(\Om)), h\in L^2(\Om).$  Let $u_\va, u_0\in L^2(0,T; H^m(\Om))\cap L^\infty(0,T; L^2(\Om))$ be the weak solutions to problem  (\ref{eq1})  and problem  (\ref{hoeq2}), respectively. Then as $\va\longrightarrow 0$,
\begin{align}\label{t21re1}
u_\va \longrightarrow u_0  \text{ weakly in } L^2(0,T; H^m(\Om)) \text{ and strongly in }  L^2(0,T; H^{m-1}(\Om)).
\end{align}
\end{theorem}

\section{Smoothing operators and auxiliary estimates }
We fix nonnegative functions $ \varphi_1(s)\in C_c^\infty(-1/2, 1/2), \,\varphi_2(y)\in C_c^\infty(B(0,1/2)),$ such that $$\int_{\R} \varphi_1(s)=1  \quad\text{and}\quad  \int_{\mathbb{R}^{d}} \varphi_2(y) =1.$$
Set   $\varphi_{1,\varepsilon} (s)=\frac{1}{\varepsilon^{2m}} \varphi_1(s/\va^{2m}), \,\varphi_{2,\varepsilon}(y)=\frac{1}{\varepsilon^{d}} \varphi_2(y/\va),$   and define
\begin{align}\label{sva}
\begin{split}
&\widetilde{S}_\varepsilon(f)(x,t)=\int_{\mathbb{R}^{d}}  \varphi_{2,\va} (y)f(x-y, t)\, dy=\int_{\mathbb{R}^d}\varphi_2(y)f(x-\va y, t)  dy,\\
&S_\varepsilon(f)(x,t)=\int_{\mathbb{R}^{d+1}} \varphi_{1,\va} (s) \varphi_{2,\va} (y)f(x-y, t-s)=\int_{\mathbb{R}^{d+1}} \varphi_1(s) \varphi_2(y)f(x-\va y, t-\va^{2m} s). \\
\end{split}
\end{align}
By the definition, it is obvious that
\begin{align}\label{wtsva}
S_\va(f)(x,t)=\int_{\R} \varphi_{1,\va}(s) \wt{S}_\va(f)(x,t-s)ds
\end{align}

\begin{lemma}\label{le2.2}
Let $f\in L^p(\mathbb{R}^{d+1})$ for some $1\!\leq \!p<\!\infty$, $g\in L_{loc}^p(\mathbb{R}^{d+1})$ and  $h\in L^\infty(\mathbb{R}^{d+1})$ with support $\mathcal{O}$. Then
\begin{align*}
\|g^\varepsilon S_\varepsilon(f) h \|_{L^p(\mathbb{R}^{d+1})}
\leq C \sup_{z\in \mathbb{R}^{d+1}}\Big( \fint_{B(z,1)} |g|^p  \Big)^{1/p} \|f\|_{L^p(\mathcal{O}(\varepsilon) )}\| h\|_\infty,
\end{align*}
where $g_\varepsilon (x,t)=g(x/\va, t/\va^{2m})$,
$\mathcal{O}(\varepsilon)=\big\{z\in \mathbb{R}^{d+1}: dist(z,\mathcal{O})<\varepsilon \big\}$.  If in addition  $g(y,s)$ is 1-periodic in $(y,s)$, then
\begin{align*}
\|g_\varepsilon S_\varepsilon(f) h \|_{L^p(\mathbb{R}^{d+1})}\leq C   \|g\|_{L^p(Y)}\|f\|_{L^p(\mathcal{O}(\varepsilon) )} \| h\|_\infty.
\end{align*}
\end{lemma}
\begin{proof}
See \cite[Lemma 3.3]{gsjfa2017} and also \cite[Lemma 2.1]{shenan2017}.
\end{proof}

For $k>0$, let  $ \Omega_{ k\varepsilon}=\{x\in\Omega: dist(x, \partial \Omega)\leq k\varepsilon\}$ and
 $$\Om_{T, k\varepsilon}= \Omega_{k\varepsilon}\times (0,T) \cup \Om\times  (0, k\va^{2m}]\cup \Om\times[T-k\va^{2m}, T).$$
\begin{lemma}\label{le2.3}
Let  $S_\va(f)$ be defined as in (\ref{sva}). Then for any $1\!<\!q\!<\infty,$ and any  integer $\ell>0,$  \begin{align}
& \|\pa_t S_\varepsilon(f)\|_{L^q(\Om_T\setminus \Om_{T,2\va})}\leq C \varepsilon^{-2m}\|f\|_{L^q(\Om_T\setminus \Om_{T,\va})},\label{le23re1}\\
&  \|\na^\ell S_\varepsilon(f)\|_{L^q(\Om_T\setminus \Om_{T,2\va} )}\leq C \varepsilon^{-\ell} \|f\|_{L^q(\Om_T\setminus \Om_{T,\va})},\label{le23re2}\\
& \|  S_\varepsilon(\na^\ell f) - \na^\ell  f\|_{L^q(\Om_T\setminus \Om_{T,4\va})}\leq C \va  \|\nabla^{\ell+1} f\|_{L^q(\Om_T\setminus \Om_{T,2\va})}\nonumber\\
& \quad\quad\quad\quad\quad\quad\quad\quad\quad\quad\quad\quad \quad  + C   \va^{m-\ell+1}\|\pa_t f\|_{L^q(\va^{2m},T-\va^{2m}; W^{-m+1, q}(\Om\setminus\Om_{2\va}))},\label{le23re3}
\end{align}
where $C$ depends only on $d, q, \Om, T.$
\end{lemma}
\begin{proof}  Note that
\begin{align}\label{ple231}
\|\pa_tS_\varepsilon(f )\|^q_{L^q(\Om_T\setminus \Om_{T,2\va})} &=\int_{\Om_T\setminus \Om_{T,2\va}}  \Big|\int_{\mathbb{R}^{d+1}}
\pa_t \varphi_\varepsilon (x-y, t-s)  f(y,s)     \, dyds\Big|^q   \,dx dt,
\end{align}
where $\varphi_\varepsilon (x-y, t-s)=\varphi_{1,\va}(t-s)\varphi_{2,\va}(x-y)$.
Using H\"{o}lder's inequality, we deduce that
\begin{align}\label{ple232}
 &\Big|\int_{\mathbb{R}^{d+1}}
\pa_t \varphi_\varepsilon (x-y, t-s)  f(y,s)     \, dy ds\Big|^q  \nonumber\\
&\leq   \Big\{\int_{\mathbb{R}^{d+1}}
|\pa_t \varphi_\varepsilon (x-y, t-s)|\, |f(y,s)|^q \,dyds \Big\} \Big\{\int_{\mathbb{R}^{d+1}}
|\pa_t \varphi_\varepsilon (x-y, t-s)|  \,dyds \Big\}^{q-1} \nonumber\\
&\leq C \va^{-2(q-1)m }   \Big\{\int_{\mathbb{R}^{d+1}}
|\pa_t \varphi_\varepsilon (x-y, t-s)|\, |f(y,s)|^q \,dyds \Big\}.
\end{align}
Taking (\ref{ple232}) into (\ref{ple231}) and using Fubini's theorem, it yields
\begin{align*}
\|\pa_t S_\varepsilon( f )\|^q_{L^q(\Om_T\setminus \Om_{T,2\va})}  &\leq C \varepsilon^{-2mq}
\int_{\Om_T\setminus \Om_{T,\va}} |f(y,s) |^q dyds,
\end{align*} which is exactly (\ref{le23re1}). The proof of (\ref{le23re2}) is similar and we therefore pass to (\ref{le23re3}).

Recall that (\ref{le23re3}) was essentially proved in  \cite[Lemma 3.2]{gsjfa2017} for the case $m=\ell=1, q=2$ by the Plancherel Theorem, which is not applicable for general $q$. By (\ref{wtsva}),
\begin{align}\label{ple234}
 \|S_\varepsilon(\na^\ell f)- \na^\ell f\|^q_{L^q(\Om_T\setminus \Om_{T,4\va})}
&\leq \int_{\Om_T\setminus \Om_{T,4\va}} \Big|\int_{\R}\varphi_1(s) \wt{S}_\va( \na^\ell f)(x,t-\va^{2m}s) ds -\wt{S}_\va(\na^\ell f)(x,t)\Big|^q \nonumber\\
& \quad  + \int_{\Om_T\setminus \Om_{T,4\va}} \big|\wt{S}_\va(\na^\ell f)(x,t) -\na^\ell f(x,t)\big|^q \doteq (I)+(II).
 \end{align}
Similar to (\ref{ple232}), we may use H\"{o}lder's inequality and Fubini's theorem to deduce that
\begin{align}\label{ple235}
(II)
 \leq C \int_{\Om_T\setminus \Om_{T,4\va}}  \int_{\R^d}\varphi_2(y)  \big|\na^\ell f(x-\va y, t)-\na^\ell f(x,t)\big|^q   dy\,dxdt
 \leq C  \va^q \|\na^{\ell+1} f\|^q_{L^q(\Om_T\setminus \Om_{T,2\va})}.
\end{align}
And also,
\begin{align}\label{ple236}
\begin{split}
(I)
&\leq \int_{\Om_T\setminus \Om_{T,4\va}}  \int_{\R}  \varphi_1(s)  \big|\wt{S}_\va( \na^\ell f)(x,t-\va^{2m}s)   -\wt{S}_\va(\na^\ell f)(x,t) \big|^q ds\, dxdt \\
&\leq C \va^{2mq}  \big\| \pa_t \wt{S}_\va(\na^\ell f)\big\|^q_{L^q(\Om_T\setminus \Om_{T,2\va})} \int_{\R}  \varphi_1(s) ds \\
&\leq C \va^{(m+1-\ell)q}  \| \pa_t  f\|^q_{L^q(\va^{2m},T-\va^{2m}; W^{-m+1, q}(\Om\setminus\Om_\va))},\end{split}
\end{align} which, together with (\ref{ple234}) and (\ref{ple235}), gives (\ref{le23re3}).
For the last step of (\ref{ple236}), we have used the following observation
\begin{align}\label{ple237}
\| \pa_t \wt{S}_\va(\na^\ell f)\|_{L^q(\Om_T\setminus \Om_{T,2\va})} \leq C \va^{-m-\ell+1}  \| \pa_t  f\|_{L^q(\va^{2m},T-\va^{2m}; W^{-m+1, q}(\Om\setminus\Om_\va))}.
\end{align}
To see this, we note that for any $g\in C_c^\infty(\Om_T\!\setminus\!\Om_{T,2\va})$ and $d$-dimensional multi-index $\eta$ with $|\eta|=\ell$,
\begin{align*}
 &\int_{\Om_T\setminus \Om_{T,2\va} }\pa_t \wt{S}_\va(\na^\ell f)(x,t)g(x,t)\,dxdt \nonumber\\
 &= \int_{2\va^{2m}}^{T-2\va^{2m}}\int_{\Om\setminus\Om_{2\va}}\pa_t\Big\{\int_{\R^d} \varphi_{2,\va}(x-y) D_y^\eta f(y,t) \, dy\Big\} g(x,t)dxdt \nonumber\\
&=(-1)^{\ell+1} \int_{2\va^{2m}}^{T-2\va^{2m}}\!\Big\{\int_{\Om\setminus\Om_{2\va}} \int_{\R^d} D_x^\eta\varphi_{2,\va}(x-y) f(y,t)  \pa_t g(x,t)\, dy dx \Big\} dt \nonumber\\
&=(-1)^{\ell+1}\int_{2\va^{2m}}^{T-2\va^{2m}}\!\Big\{\int_{\Om\setminus\Om_{\va}}  f(y,t) \pa_t\Big[\int_{\R^{d}} D_x^\eta \varphi_{2,\va}(x-y) g(x,t)dx\Big] dy\Big\} dt  \nonumber\\
&=(-1)^{\ell}\int_{ 2\va^{2m}}^{T- 2\va^{2m}}\!\big\langle \pa_t f(y,t),  D^\eta\wt{S}_\va (g)\big\rangle_{W^{-m+1,q}(\Om\setminus\Om_{\va})\times W_0^{m-1,q^{\prime}}(\Om\setminus\Om_{\va})}  dt, \,\quad q^{\prime}=q/(q-1),  \end{align*}
where  the integration by parts and Fubini's theorem have been used  for the second  and the third equalities respectively.
Therefore, using estimates similar to (\ref{le23re2}) we may deduce that
\begin{align*}
\big| \int_{\Om_T\setminus \Om_{T,2\va} }\pa_t \wt{S}_\va(\na^\ell f)(x,t)g(x,t) \big|&\leq C  \| \pa_t  f\|_{L^q(\va^{2m},T-\va^{2m}; W^{-m+1, q}(\Om\setminus\Om_{\va}))} \| \na^{m+\ell-1}\wt{S}_\va(g)\|_{L^{q^{\prime}}(\Om_T\setminus \Om_{T,\va})}\\
&\leq C\va^{-m-\ell+1} \| \pa_t  f\|_{L^q(\va^{2m},T-\va^{2m}; W^{-m+1, q}(\Om\setminus\Om_{\va}))} \|g\|_{L^{q^{\prime}}(\Om_T\setminus \Om_{T,2\va})},
\end{align*}
which implies (\ref{ple237}) directly.  The proof is complete.
\end{proof}

\section{Convergence rates for the initial-Dirichlet problem}
\subsection{ $O(\sqrt{\va})$ error estimate in $L^2(0,T; H^{m}(\Om))$}
Let $\rho_\varepsilon$ be a function in $C_c^\infty(\Omega)$ such that
\begin{align*}
 supp(\rho_\varepsilon)\subset \Om\setminus\Omega_{6\varepsilon},\quad
 0\leq\rho_\varepsilon\leq 1,\quad\rho_\varepsilon=1  \text{ on } \Om\setminus\Omega_{8\varepsilon}
 \quad\text{and}\quad |\nabla^k \rho_\varepsilon|\leq C \varepsilon^{-k} \text{ for } 1\le k\le m.
 \end{align*}
For $0<T<\infty$, let $\varrho_\varepsilon$ be a function in $C_c^\infty(0,T)$  with $ supp(\varrho_\varepsilon) \subseteq (6\va^{2m}, T-6\va^{2m}) $ and
\begin{align*}
0\leq\varrho_\varepsilon\leq 1, \quad
 \varrho_\varepsilon=1  \text{ in } (8\va^{2m}, T-8\va^{2m}), \quad
 |   \varrho'_\varepsilon|\le C \varepsilon^{-2m} .
\end{align*}
Define
\begin{align}\label{w}
\varpi_\varepsilon(x, t)=u_\varepsilon(x, t)-u_0(x, t)-\varepsilon^m\sum_{|\gamma|=m}\chi^\gamma_\varepsilon(x, t)K_\varepsilon(D^\gamma u_0)(x, t),
\end{align}
where $K_\varepsilon(u)=S^2_\varepsilon(u)\rho_\varepsilon\varrho_\varepsilon$, $S^2_\varepsilon=S_\varepsilon\circ S_\varepsilon$ and $\chi_\varepsilon^\gamma(x, t)=\chi^\gamma(x/\varepsilon, t/\varepsilon^{2m})$.

\begin{lemma}\label{le3.1}
  Let $\Omega$ be a bounded Lipschitz domain in $\mathbb{R}^d$, and $A$ satisfy conditions (\ref{cod1}) and (\ref{cod2}). Let $u_\varepsilon, u_0$ be weak solutions to initial-Dirichlet problems (\ref{eq1}) and (\ref{hoeq1}), respectively. Then we have
\begin{align}
(\pa_t+\mc{L}_\va ) \varpi_\va&
=(-1)^{m+1}\sum_{|\al|=|\be|=m}  D^\al\Big\{\big(A_\va^{\alpha \beta} -\bar{A}^{\alpha \beta}\big)
 \big(D^\beta u_0-K_\varepsilon(D^\beta u_0)\big)\Big\}\nonumber\\
& \quad+ (-1)^{m+1}\sum_{\substack{|\alpha|=|\beta|=|\gamma|=m\\ \zeta+\eta =\gamma\\  0\leq|\zeta|\leq m-1}}
  C(\gamma, \zeta)\varepsilon^{|\eta|}  D^\alpha  \Big\{\big[A_\va^{\alpha \gamma}   (D^{\zeta}\chi^\beta)_\varepsilon-(D^\zeta\mathfrak{B}^{\gamma \alpha \beta})_\varepsilon\big]D^\eta K_\varepsilon(D^\beta u_0) \Big\}\nonumber\\
  &\quad +(-1)^m\varepsilon^{2m}\sum_{|\al|=|\be|=m}D^\alpha\Big\{\big(\mathfrak{B}_\varepsilon^{(d+1)\alpha\beta}+\mathcal{B}_\varepsilon^{\alpha\beta}\big)\partial_tK_\varepsilon(D^\beta u_0) \Big\}\nonumber\\
                              &\quad +(-1)^m \sum_{\substack{|\beta|=|\gamma|=m\\ \zeta+\eta =\gamma\\0\leq|\zeta|\leq m-1}}\varepsilon^{m+|\eta|}\pa_t\Big\{(D^\zeta\mathfrak{B}^{\gamma (d+1) \beta})_\varepsilon D^\eta K_\varepsilon(D^\beta u_0 )\Big\}\nonumber\\
  &\quad+(-1)^{m+1}\varepsilon^{2m}\sum_{|\alpha|=|\beta|=m}\partial_t\Big\{\mathcal{B}_\varepsilon^{\alpha\beta}D^\alpha K_\varepsilon(D^\beta u_0)\Big\}.\label{lem41}
\end{align}
\end{lemma}

\begin{proof}
For the simplicity of presentations, let us omit the subscripts $i,j$ in (\ref{le31re1}) and hereafter.
Using definitions of $\varpi_\varepsilon$ and $B^{\al\be}$ (see (\ref{w}) and (\ref{duc}) respectively),
a direct computation yields that
\begin{align}
 (\pa_t+\mc{L}_\va ) \varpi_\va&
=(-1)^{m+1}\sum_{|\al|=|\be|=m}  D^\al\Big\{\big(A_\va^{\alpha \beta} -\bar{A}^{\alpha \beta}\big)
 \big(D^\beta u_0-K_\varepsilon(D^\beta u_0)\big)\Big\} \nonumber\\
 &\quad+ (-1)^{m+1}\varepsilon^{m}  \sum_{\substack{|\alpha|=|\beta|=|\gamma|=m\\ \zeta+\eta =\beta\\  0\leq|\zeta|\leq m-1}}
C(\alpha, \zeta)  D^\alpha  \left\{A_\va^{\alpha \beta}   D^{\zeta}\chi_\va^\gamma   D^{\eta}  \big[K_\va(D^\gamma u_0 )\big]\right\} \nonumber\\
& \quad+ (-1)^{m+1}\sum_{|\al|=|\be|=m}D^\al\Big\{B_\va^{\alpha \beta}   K_\varepsilon(D^\beta u_0) \Big\}
 - \va^m   \sum_{|\ga|=m} \pa_t \big[ \chi_\va^\gamma K_\varepsilon(D^\gamma u_0 )\big],\label{iden_w}
\end{align}
where $  B_\va^{\al\be}(x, t)=  B^{\al\be}(x/\va, t/\va^{2m})$ and $C(\alpha, \zeta)=\frac{\alpha!}{\zeta!(\alpha-\zeta)!}$.
In view of \eqref{le21re1}, we may deduce that
\begin{align}
&(-1)^{m+1}\sum_{|\al|=|\be|=m}D^\al\Big\{B_\va^{\alpha \beta}   K_\varepsilon(D^\beta u_0) \Big\}
                - \va^m   \sum_{|\ga|=m} \pa_t \big[ \chi_\va^\gamma K_\varepsilon(D^\gamma u_0 )\big]\nonumber\\
&=(-1)^{m+1}\sum_{|\al|=|\be|=m}D^\al\Big\{\big(\sum_{|\gamma|=m}\varepsilon^mD^\gamma\mathfrak{B}_\varepsilon^{\gamma \alpha \beta}+\varepsilon^{2m}\partial_t\mathfrak{B}_\varepsilon^{(d+1)\alpha\beta}+\widehat{B}_\varepsilon^{\alpha\beta}\big) K_\varepsilon(D^\beta u_0) \Big\}\nonumber\\
  &\quad +(-1)^{m+1} \va^m \sum_{|\beta|=m} \pa_t \Big\{\big(\sum_{|\gamma|=m}\varepsilon^mD^\gamma\mathfrak{B}_\varepsilon^{\gamma (d+1) \beta}\big) K_\varepsilon(D^\beta u_0 )\Big\}\nonumber\\
 & =(-1)^{m+1}\varepsilon^m\sum_{|\al|=|\be|=|\gamma|=m}D^\al D^\gamma\Big\{\mathfrak{B}_\varepsilon^{\gamma \alpha \beta}K_\varepsilon(D^\beta u_0) \Big\}\nonumber\\
 &\quad+(-1)^m\varepsilon^m\sum_{\substack{|\alpha|=|\beta|=|\gamma|=m\\ \zeta+\eta =\gamma\\  0\leq|\zeta|\leq m-1}}
  C(\gamma, \zeta)  D^\alpha  \Big\{D^\zeta\mathfrak{B}_\varepsilon^{\gamma \alpha \beta}D^\eta K_\varepsilon(D^\beta u_0) \Big\}\nonumber\\
&\quad+(-1)^{m+1}\varepsilon^{2m}\sum_{|\al|=|\be|=m}D^\alpha\partial_t\Big\{\mathfrak{B}_\varepsilon^{(d+1)\alpha\beta}K_\varepsilon(D^\beta u_0) \Big\}\nonumber\\
&\quad+(-1)^m\varepsilon^{2m}\sum_{|\al|=|\be|=m}D^\alpha\Big\{\mathfrak{B}_\varepsilon^{(d+1)\alpha\beta}\partial_tK_\varepsilon(D^\beta u_0) \Big\}\nonumber\\
  &\quad +(-1)^{m+1} \va^{2m} \sum_{|\beta|=|\gamma|=m} \pa_tD^\gamma \Big\{\mathfrak{B}_\varepsilon^{\gamma (d+1) \beta}K_\varepsilon(D^\beta u_0 )\Big\}\nonumber\\
  &\quad+(-1)^m \va^{2m} \sum_{\substack{|\beta|=|\gamma|=m\\ \zeta+\eta =\gamma\\0\leq|\zeta|\leq m-1}} \pa_t\Big\{D^\zeta\mathfrak{B}_\varepsilon^{\gamma (d+1) \beta}D^\eta K_\varepsilon(D^\beta u_0 )\Big\}\nonumber\\
&\quad+(-1)^{m+1}\sum_{|\alpha|=|\beta|=m}D^\alpha\Big\{\widehat{B}_\varepsilon^{\alpha\beta}K_\varepsilon(D^\beta u_0)\Big\}\nonumber\\
  &=(-1)^m\varepsilon^m\sum_{\substack{|\alpha|=|\beta|=|\gamma|=m\\ \zeta+\eta =\gamma\\  0\leq|\zeta|\leq m-1}}
  C(\gamma, \zeta)  D^\alpha  \Big\{D^\zeta\mathfrak{B}_\varepsilon^{\gamma \alpha \beta}D^\eta K_\varepsilon(D^\beta u_0) \Big\}\nonumber\\
  &\quad+(-1)^m\varepsilon^{2m}\sum_{|\al|=|\be|=m}D^\alpha\Big\{\mathfrak{B}_\varepsilon^{(d+1)\alpha\beta}\partial_tK_\varepsilon(D^\beta u_0) \Big\}\nonumber\\
  &\quad+(-1)^m \va^{2m} \sum_{\substack{|\beta|=|\gamma|=m\\ \zeta+\eta =\gamma\\0\leq|\zeta|\leq m-1}} \pa_t\Big\{D^\zeta\mathfrak{B}_\varepsilon^{\gamma (d+1) \beta}D^\eta K_\varepsilon(D^\beta u_0 )\Big\}\nonumber\\
  &\quad+(-1)^{m+1}\sum_{|\alpha|=|\beta|=m}D^\alpha\Big\{\widehat{B}_\varepsilon^{\alpha\beta}K_\varepsilon(D^\beta u_0)\Big\},\label{iden_1}
\end{align}
where we have used the facts $\mathfrak{B}^{(d+1)(d+1)\beta}=0$ and $\widehat{B}^{(d+1)\,\beta}=0$ for the second step and the skew-symmetry of $\mathfrak{B}$ for the last step. Substituting \eqref{iden_1} into \eqref{iden_w}, we get
\begin{align}
(\pa_t+\mc{L}_\va ) \varpi_\va&
=(-1)^{m+1}\sum_{|\al|=|\be|=m}  D^\al\Big\{\big(A_\va^{\alpha \beta} -\bar{A}^{\alpha \beta}\big)
 \big(D^\beta u_0-K_\varepsilon(D^\beta u_0)\big)\Big\}\nonumber\\
& \quad+ (-1)^{m+1}\sum_{\substack{|\alpha|=|\beta|=|\gamma|=m\\ \zeta+\eta =\gamma\\  0\leq|\zeta|\leq m-1}}
  C(\gamma, \zeta)\varepsilon^{|\eta|}  D^\alpha  \Big\{\big[A_\va^{\alpha \gamma}   (D^{\zeta}\chi^\beta)_\varepsilon-(D^\zeta\mathfrak{B}^{\gamma \alpha \beta})_\varepsilon\big]D^\eta K_\varepsilon(D^\beta u_0) \Big\}\nonumber\\
  &\quad +(-1)^m\varepsilon^{2m}\sum_{|\al|=|\be|=m}D^\alpha\Big\{\mathfrak{B}_\varepsilon^{(d+1)\alpha\beta}\partial_tK_\varepsilon(D^\beta u_0) \Big\}\nonumber\\
  &\quad +(-1)^m \sum_{\substack{|\beta|=|\gamma|=m\\ \zeta+\eta =\gamma\\0\leq|\zeta|\leq m-1}}\varepsilon^{m+|\eta|}\pa_t\Big\{(D^\zeta\mathfrak{B}^{\gamma (d+1) \beta})_\varepsilon D^\eta K_\varepsilon(D^\beta u_0 )\Big\}\nonumber\\
  &\quad +(-1)^{m+1}\sum_{|\alpha|=|\beta|=m}D^\alpha\Big\{\widehat{B}_\varepsilon^{\alpha\beta}K_\varepsilon(D^\beta u_0)\Big\}.\label{lem41_1}
\end{align}

Moreover, thanks to Lemma \ref{lem22},
\begin{align*} (-1)^{m+1}\sum_{|\alpha|=|\beta|=m}D^\alpha\Big\{\widehat{B}_\varepsilon^{\alpha\beta}K_\varepsilon(D^\beta u_0)\Big\}&=(-1)^{m+1}\sum_{|\alpha|=|\beta|=m}D^\alpha\Big\{\varepsilon^{2m}\partial_t\mathcal{B}_\varepsilon^{\alpha\beta}K_\varepsilon(D^\beta u_0)\Big\}\nonumber\\
  &=(-1)^{m+1}\varepsilon^{2m}\sum_{|\alpha|=|\beta|=m}\partial_t\Big\{\mathcal{B}_\varepsilon^{\alpha\beta}D^\alpha K_\varepsilon(D^\beta u_0)\Big\}\nonumber\\ &\quad+(-1)^m\varepsilon^{2m}\sum_{|\alpha|=|\beta|=m}D^\alpha\Big\{\mathcal{B}_\varepsilon^{\alpha\beta} \partial_tK_\varepsilon(D^\beta u_0)\Big\}.
\end{align*}
This, together with \eqref{lem41_1}, gives \eqref{lem41}.
\end{proof}

In the following, we define
\begin{align}
w_\varepsilon(x, t)&=u_\varepsilon(x, t)-u_0(x, t)-\varepsilon^m\sum_{|\gamma|=m}\chi^\gamma_\varepsilon(x, t)K_\varepsilon(D^\gamma u_0)(x, t)\nonumber\\ &\quad +(-1)^{m+1} \sum_{\substack{|\beta|=|\gamma|=m\\ \zeta+\eta =\gamma\\0\leq|\zeta|\leq m-1}}\varepsilon^{m+|\eta|}(D^\zeta\mathfrak{B}^{\gamma (d+1) \beta})_\varepsilon D^\eta K_\varepsilon(D^\beta u_0) \nonumber\\
  &\quad +(-1)^m\varepsilon^{2m}\sum_{|\alpha|=|\beta|=m}\mathcal{B}_\varepsilon^{\alpha\beta}D^\alpha K_\varepsilon(D^\beta u_0).\label{def_w}
\end{align}

\begin{lemma}\label{le3.2}
In addition of  the assumptions of Lemma \ref{le3.1}, if $u_0\in L^2(0,T; H^{m+1}(\Om))$  we have for any $\phi \in L^2(0,T; H^m_0(\Omega))$,
\begin{align}\label{le32re2}
&\Big|\int_0^T \big\langle \pa_t w_\va, \phi\big\rangle  dt +  \sum_{|\alpha|=|\beta|=m} \int_{\Om_T} D^\alpha \phi A_\va^{\alpha\beta}
 D^\beta w_{\varepsilon}\Big| \nonumber\\
 &\leq  C \Big\{\|u_0\|_{L^2(0,T; H^{m+1}(\Om))}+ \|\pa_t  u_0\|_{L^2(0,T; H^{-m+1}(\Om))}+ \sup_{10\va^{2m}< t< T}\Big(\frac{1}{\va}\int^t_{t-10\va^{2m}} \| \nabla^m  u_0(t)\|^2_{L^2(\Om)} \Big)^{1/2} \Big\} \nonumber\\ &\quad  \times \left\{\varepsilon \|\nabla^m \phi\|_{L^2(\Omega_T)}  +\varepsilon^{1/2}  \|\nabla^m \phi\|_{L^2(\Omega_{T,8\varepsilon})} \right \},
\end{align}where $C$ depends only on $d, n, m, \mu, T$ and $\Omega.$
\end{lemma}
\begin{proof}
  According to \eqref{lem41}, we have
\begin{align*}
(\pa_t+\mc{L}_\va ) w_\va&
=(-1)^{m+1}\sum_{|\al|=|\be|=m}  D^\al\Big\{\big(A_\va^{\alpha \beta} -\bar{A}^{\alpha \beta}\big)
 \big(D^\beta u_0-K_\varepsilon(D^\beta u_0)\big)\Big\}\nonumber\\
& \quad+ (-1)^{m+1}\sum_{\substack{|\alpha|=|\beta|=|\gamma|=m\\ \zeta+\eta =\gamma\\  0\leq|\zeta|\leq m-1}}
  C(\gamma, \zeta)\varepsilon^{|\eta|}  D^\alpha  \Big\{\Big(A_\va^{\alpha \gamma}   (D^{\zeta}\chi^\beta)_\varepsilon-(D^\zeta\mathfrak{B}^{\gamma \alpha \beta})_\varepsilon\Big)D^\eta K_\varepsilon(D^\beta u_0) \Big\}\nonumber\\
  &\quad-\sum_{|\alpha|=|\xi|=m}D^\alpha\Big\{A^{\alpha\xi}_\varepsilon D^\xi\Big[ \sum_{\substack{|\beta|=|\gamma|=m\\ \zeta+\eta =\gamma\\0\leq|\zeta|\leq m-1}}\varepsilon^{m+|\eta|}(D^\zeta\mathfrak{B}^{\gamma (d+1) \beta})_\varepsilon D^\eta K_\varepsilon(D^\beta u_0)\Big]\Big\} \nonumber\\
  &\quad +\sum_{|\alpha|=|\xi|=m}D^\alpha\Big\{A_\varepsilon^{\alpha\xi}D^\xi\Big(\varepsilon^{2m}\sum_{|\gamma|=|\beta|=m}\mathcal{B}_\varepsilon^{\gamma\beta}D^\gamma K_\varepsilon(D^\beta u_0)\Big)\Big\}\nonumber\\ &\quad+(-1)^m\varepsilon^{2m}\sum_{|\al|=|\be|=m}D^\alpha\Big\{\Big(\mathfrak{B}_\varepsilon^{(d+1)\alpha\beta}+\mathcal{B}_\varepsilon^{\alpha\beta}\Big)\partial_tK_\varepsilon(D^\beta u_0) \Big\},
\end{align*}
which yields that for any $\phi \in L^2(0, T; H^m_0(\Omega))$
\begin{align}
  &  \int_0^T \langle \pa_t w_\va, \phi\rangle  dt+\sum_{|\alpha|=|\beta|=m}  \int_{\Om_T}
A_\va^{\alpha\beta}   D^\beta w_\varepsilon D^\alpha \phi\nonumber\\
 &= - \sum_{|\alpha|=|\beta|=m}\int_{\Om_T}
 \big(A_\va^{\alpha \beta} -\bar{A}^{\alpha \beta}\big)
 \big(D^\beta u_0-K_\varepsilon(D^\beta u_0)\big) D^\alpha \phi\nonumber\\
&\ \ \ - \sum_{\substack{|\alpha|=|\beta|=|\gamma|=m\\ \zeta+\eta =\gamma\\  0\leq|\zeta|\leq m-1}}\varepsilon^{|\eta|}
C(\gamma, \zeta)  \int_{\Om_T}\Big(A_\va^{\alpha \gamma}   (D^{\zeta}\chi^\beta)_\varepsilon-(D^\zeta\mathfrak{B}^{\gamma \alpha \beta})_\varepsilon\Big)D^\eta K_\varepsilon(D^\beta u_0) D^\alpha \phi \nonumber\\
&\quad+(-1)^{m+1}\sum_{\substack{|\alpha|=|\xi|=|\beta|=|\gamma|=m\\ \zeta+\eta =\gamma+\xi\\0\leq|\zeta|\leq 2m-1}}\varepsilon^{|\eta|} C(\zeta)\int_{\Om_T}A^{\alpha\xi}_\varepsilon(D^\zeta\mathfrak{B}^{\gamma (d+1) \beta})_\varepsilon D^\eta K_\varepsilon(D^\beta u_0)D^\al \phi\label{le31re1}\\
  &\quad +(-1)^m\varepsilon^{2m}\sum_{|\alpha|=|\xi|=|\beta|=|\gamma|=m}\int_{\Om_T}A_\varepsilon^{\alpha\xi}\mathcal{B}_\varepsilon^{\gamma\beta}D^\xi D^\gamma K_\varepsilon(D^\beta u_0)D^\alpha\phi\nonumber\\
&\ \ \ + \va^{2m} \sum_{|\al|=|\be|=m} \int_{\Om_T} \Big(\mathfrak{B}_\varepsilon^{(d+1)\alpha\beta}+\mathcal{B}_\varepsilon^{\alpha\beta}\Big)\partial_tK_\varepsilon(D^\beta u_0) D^\al \phi.\nonumber
\end{align}
  Denote the terms in the r.h.s of (\ref{le31re1}) as $I_i$, $i=1, \dots, 5$, in turn.
Note that
\begin{align*}
 & \big| D^\beta u_0-K_\varepsilon(D^\beta u_0)\big|\\
  &\leq \big|\big[D^\beta u_0- S_\varepsilon(D^\beta u_0)\big]\rho_\varepsilon \varrho_\va \big|+\big|\big[S_\varepsilon(D^\beta u_0) - S^2_\varepsilon(D^\beta u_0)\big] \rho_\varepsilon \varrho_\va\big|+\big| D^\beta u_0   (1- \rho_\varepsilon \varrho_\va)\big|,
\end{align*}
which implies that
\begin{align*}
|I_1| &\leq C   \| \nabla^m \phi\|_{L^2(\Omega_T)} \| \na^{m} u_0-S_\varepsilon(\na^m u_0)\|_{L^{2}(\Omega_T\setminus \Omega_{T,4\va})}
 +C\| \nabla^m \phi\|_{L^2(\Omega_{T,8\va})} \| \nabla^m  u_0 \|_{L^2(\Omega_{T, 8\varepsilon})}.
 \end{align*}
To deal with $I_2,$ we observe that for $|\eta|\geq1$,
\begin{align}
  D^\eta K_\varepsilon(D^\beta u_0)=&\sum_{\eta'+\eta''=\eta} C(\eta,\eta')D^{\eta'}S^2_\varepsilon(D^\beta u_0)D^{\eta''}\rho_\varepsilon\varrho_\varepsilon\nonumber\\
 =&\sum_{\substack{\eta'+\eta''=\eta\\1\leq |\eta''| }} C(\eta,\eta')D^{\eta'}S^2_\varepsilon(D^\beta u_0)D^{\eta''}\rho_\varepsilon\varrho_\varepsilon+S^2_\varepsilon(D^\beta D^{\eta} u_0)\rho_\varepsilon\varrho_\varepsilon,
\end{align}
which implies by Lemma \ref{le2.2} that
\begin{align}\label{esj2}
|I_2 | &\leq C  \sum_{1\leq k\leq m}\va^{k }  \| \nabla^m \phi\|_{L^2(\Omega_T)} \|  S_\va(\nabla^{m+k}  u_0)\|_{L^2(\Omega_T\setminus \Omega_{T, 4\va})}\nonumber\\
&\quad+ C   \sum_{ 0\leq k \leq m-1 }\va^{k } \| \nabla^m \phi\|_{L^2(\Omega_{T, 8\va})} \| S_\va(\nabla^{m+k} u_0 )\|_{L^2(\Omega_{T, 9\va}\setminus \Omega_{T,4\va})},
\end{align}
where $C$ depends only on $d, n, m, \mu, T$ and $\Omega$.
In a similar manner, we also have
\begin{align}\label{esj3}
|I_3 |+|I_4| &\leq C  \sum_{1\leq k\leq 2m}\va^{k }  \| \nabla^m \phi\|_{L^2(\Omega_T)} \|  S_\va(\nabla^{m+k}  u_0)\|_{L^2(\Omega_T\setminus \Omega_{T, 4\va})}\nonumber\\
&\quad+ C   \sum_{ 0\leq k \leq 2m-1 }\va^{k } \| \nabla^m \phi\|_{L^2(\Omega_{T, 8\va})} \| S_\va(\nabla^{m+k} u_0 )\|_{L^2(\Omega_{T, 9\va}\setminus \Omega_{T,4\va})},
\end{align}
where $C$ depends only on $d, n, m, \mu, T$ and $\Omega$.

Finally,  by Lemma \ref{le2.2}  it is not difficult to find that
\begin{align}\label{esj5}
|I_5|
&\leq C \va^{2m}  \| \nabla^m \phi\|_{L^2(\Omega_T)} \|   S_\varepsilon(\na^m\pa_tu_0 )\|_{L^2(\Om_T\setminus \Om_{T,4\va})}\nonumber\\
& \quad +C \| \nabla^m \phi\|_{L^2(\Omega_{T,8\va})} \|   S_\varepsilon(\na^m u_0 )\|_{L^2(\Omega_{T,9\va}\setminus \Om_{T,4\va})}.
\end{align}
Taking the estimates for $I_1$--$I_5$ into (\ref{le31re1}), we obtain that \begin{equation}
\label{le32re1}
\aligned
 & \Big| \int_0^T \big\langle \pa_t w_\va, \phi\big\rangle_{H^{-m}(\Om)\times H_0^m(\Om)} dt +  \sum_{|\alpha|=|\beta|=m} \int_{\Om_T} D^\alpha \phi   A_\va^{\alpha\beta}
 D^\beta w_{\varepsilon}dxdt\Big| \\
&\leq C\| \nabla^m \phi\|_{L^2(\Omega_{T,8\va})} \| \nabla^m  u_0 \|_{L^2(\Omega_{T, 8\varepsilon})}
 +C   \| \nabla^m \phi\|_{L^2(\Omega_T)} \| \na^{m} u_0-S_\varepsilon(\na^m u_0)\|_{L^{2}(\Omega_{T}\setminus\Omega_{T,4\va})}\\
 &\quad +C  \sum_{0\leq k\leq 2m-1} \va^{k}  \| \nabla^m \phi\|_{L^2(\Om_{T,8\va})} \|S_\varepsilon(\na^{m+k}u_0)\|_{L^2(\Om_{T,9\va}\setminus \Om_{T,4\va})} \\
&\quad  + C \sum_{1\leq k\leq  2m} \va^k  \| \nabla^m \phi\|_{L^2(\Omega_T)} \|S_\varepsilon(\na^{m+k}u_0)\|_{L^2(\Om_T\setminus \Om_{T,4\va})}\\
&\quad+C \va^{2m}  \| \nabla^m \phi\|_{L^2(\Omega_T)} \|   S_\varepsilon(\na^m\pa_tu_0 )\|_{L^2(\Om_T\setminus \Om_{T,4\va})},\\
\endaligned
\end{equation}
where $C$ depends only on $d, n, m, \mu, T$ and $\Omega.$

Denote the  terms in the r.h.s. of (\ref{le32re1}) as $ J_1, J_2, J_3, J_4, J_5$ in turn.
Similar to \cite[p.664]{shenan2017}, we may prove that
\begin{align*}
\| \nabla^m  u_0(t)\|_{L^2(\Omega_{\varepsilon})}\leq C \va^{1/2}   \|u_0(t)\|_{H^{m+1}(\Om)}\quad\text{for } a.e.\,  t\in(0,T).
\end{align*}
This leads to the following estimate
\begin{align}\label{ple3201}
\| \nabla^m  u_0(t)\|^2_{L^2(\Omega_{T,8\va})}&\leq   \int_0^T \|\nabla^m  u_0(t)\|^2_{L^2(\Om_{8\varepsilon})}dt + \Big(\int_0^{8\va^{2m}} +  \int^T_{T-8\va^{2m}} \Big) \|\nabla^m  u_0(t)\|^2_{L^2(\Omega)} dt\nonumber\\
&\leq C \va  \Big\{\int_0^T\|u_0(t)\|^2_{H^{m+1}(\Om)}dt+  \sup_{  8\va^{2m}< t< T}\frac{1}{\va}\int^t_{t-8\va^{2m}} \| \nabla^m  u_0(t) \|^2_{L^2(\Om)}  dt \Big\},
\end{align}
which implies that
\begin{align}\label{ple3202}
J_1&\leq C \va^{1/2} \| \nabla^m \phi\|_{L^2(\Omega_{T,8\va})} \Big\{\Big(\int_0^T\|u_0(t)\|^2_{H^{m+1}(\Om)}dt\Big)^{1/2}\nonumber\\ & \quad\quad\quad\quad\quad\quad\quad\quad\quad\quad\quad\quad +   \sup_{8\va^{2m}< t< T}\Big(\frac{1}{\va}\int^t_{t-8\va^{2m}} \| \nabla^m  u_0(t)\|^2_{L^2(\Om)}dt \Big)^{1/2}   \Big\}.
\end{align}
By  (\ref{le23re3}),  we obtain that
\begin{align} \label{ple3203}
 J_2  
 &\leq  C \va \| \nabla^m \phi\|_{L^2(\Omega_T)} \Big\{ \|\pa_t  u_0\|_{L^2(0,T; H^{-m+1}(\Om))} +\|\na^{m+1}u_0\|_{L^2(\Om_T)} \Big\},
\end{align} where the fact $\|\pa_t  u_0\|_{L^2(0,T; H^{-m+1}(\Om\setminus\Om_\va ))}\leq \|\pa_t  u_0\|_{L^2(0,T; H^{-m+1}(\Om ))} $ is used.
By (\ref{le23re2}) and (\ref{ple237}),  we have
\begin{align*}
&\|S_\varepsilon(\na^{m+k}u_0)\|_{L^2(\Om_{T,9\va}\setminus \Om_{T,4\va})} \leq C \va^{-k} \| \nabla^m  u_0 \|_{L^2(\Omega_{T, 10\varepsilon})},  \\ &\|S_\varepsilon(\na^{m+k}u_0)\|_{L^2(\Om_T\setminus \Om_{T,4\va})}\leq  C \va^{-k+1} \| \nabla^{m+1}  u_0 \|_{L^2(\Omega_T)},\\
&\|S_\varepsilon(\na^m\pa_t u_0)\|_{L^2(\Om_T\setminus \Om_{T,4\va})}\leq  C \va^{-2m+1} \| \pa_t  u_0 \|_{L^2(0,T; H^{-m+1}(\Omega))}.
\end{align*}
This implies that
\begin{align} \label{ple3204}
\begin{split}
 J_3  &\leq C \va^{1/2} \| \nabla^m \phi\|_{L^2(\Omega_{T,8\va})} \Big\{\Big(\int_0^T\|u_0(t)\|^2_{H^{m+1}(\Om)}dt\Big)^{1/2} \\ &\quad\quad\quad\quad\quad\quad\quad\quad\quad\quad\quad+   \sup_{10\va^{2m}< t< T}\Big(\frac{1}{\va}\int^t_{t-10\va^{2m}}\! \| \nabla^m  u_0(t)\|^2_{L^2(\Om)} dt\Big)^{1/2}\Big\},\\
  J_4  &\leq   C \va  \| \nabla^m \phi\|_{L^2(\Omega_{T})}   \|\na^{m+1}u_0\|_{L^2(\Om_T)} ,\\
  J_5  &\leq   C \va  \| \nabla^m \phi\|_{L^2(\Omega_{T})}  \| \pa_t  u_0 \|_{L^2(0,T; H^{-m+1}(\Omega))} .
 \end{split}
\end{align}
Note that  (\ref{le32re2}) follows directly from (\ref{le32re1}) and (\ref{ple3202})--(\ref{ple3204}). The proof is complete.
\end{proof}


Now we are in the position to establish the error estimate in $L^2(0, T; H^{m}(\Om))$.

\begin{theorem}\label{t3.1}
Let $\Omega$ be a bounded Lipschitz domain in $\mathbb{R}^d$, and $A$ satisfy conditions (\ref{cod1}) and (\ref{cod2}). Let $u_\varepsilon, u_0$ be weak solutions to initial-Dirichlet problems (\ref{eq1}) and (\ref{hoeq1}), respectively, and moreover $u_0\in L^2(0, T; H^{m+1}(\Om))$. Let $w_\varepsilon(x,t)$ be defined by \eqref{def_w}. Then \begin{align}\label{co31re1}
\|\na^m w_\va\|_{L^2(\Om_T)}\leq C \va^{1/2} \Big\{\|u_0\|_{L^2(0, T; H^{m+1}(\Om))}+ \|f\|_{L^2(0, T; H^{-m+1}(\Om))} +\|h\|_{L^{2}(\Om)}  \Big\}.
\end{align}
\end{theorem}

\begin{proof}
Note that $w_\va\in L^2(0,T; H^m_0(\Om))$. Taking $\phi=w_\va$ in (\ref{le32re2}), it yields
 \begin{align}\label{pco311}
\|\na^m w_\va\|_{L^2(\Om_T)}\leq C \va^{1/2}\Big\{&\|u_0\|_{L^2(0, T; H^{m+1}(\Om))}+ \|\pa_t  u_0\|_{L^2(0, T; H^{-m+1}(\Om))} \nonumber\\
& + \sup_{10\va^{2m}< t< T}\Big(\frac{1}{\va}\int^t_{t-10\va^{2m}} \|\nabla^m  u_0(t) \|_{L^2(\Om)}^2 dt \Big)^{1/2} \Big\} .
\end{align}
Since
\begin{align*}
\int_{t_1}^{t_2}\int_{\Om} \bar{A}^{\al\be}D^\al u_0 D^\be u_0 =\int_{t_1}^{t_2}\langle -\pa_tu_0+f, u_0\rangle dt\quad \text{ for } \,  0  \leq t_1, t_2\leq T.
\end{align*}
By (\ref{cod1}), we have
\begin{align}\label{pco312}
&\int_{t-10\va^{2m}}^t\|\na^m u_0(t)\|_{L^2(\Om)}^2dt
 \leq C \Big( \int_0^T \|\pa_t u_0+f(t)\|^2_{H^{-m+1}(\Om)} dt\Big)^{1/2} \Big( \int_{t-10\va^{2m}}^t\|u_0(t)\|^2_{H^{m-1}(\Om))}dt\Big)^{1/2} .
\end{align}By Gagliardo-Nirenberg inequality and Young's inequality, we have
 $$\|u\|^2_{H^{m-1}(\Om)}\leq C \|u\|^{(2m-2)/m}_{H^m(\Om)} \|u\|^{2/m}_{L^2(\Om)} \leq C\va^2 \|u\|^2_{H^m(\Om)}+C\va^{-2m+2}\|u\|^2_{L^2(\Om)}.$$ We therefore  obtain that
\begin{align}\label{pco313}
\Big(\int_{t-10\va^{2m}}^t\|u_0(t)\|^2_{H^{m-1}(\Om))}dt\Big)^{1/2} &\leq C \va \Big\{ \|u_0\|_{L^2(0,T; H^{m}(\Om))}+ \|u_0\|_{L^\infty(0,T; L^{2}(\Om))} \Big\} \nonumber\\
&\leq C \va \Big\{  \|f\|_{L^2(0,T; H^{-m+1}(\Om))} +\|h\|_{L^2(\Om)} \Big\}.
\end{align}
On the other hand, from the equation of $u_0$ we note that
\begin{align}\label{pco314}
 \|\pa_t u_0\|_{L^2(0,T; H^{-m+1}(\Om))} \leq C \Big\{  \|u_0\|_{L^2(0,T; H^{m+1}(\Om))} +\|f\|_{L^2(0,T; H^{-m+1}(\Om))} \Big \}.
\end{align}
We then conclude from (\ref{pco312})--(\ref{pco314}) that
\begin{align}\label{pco315}
&\sup_{10\va^{2m}< t< T}\Big(\frac{1}{\va}\int^t_{t-10\va^{2m}}  \| \nabla^m  u_0(t)\|_{L^2(\Om)}^2 dt \Big)^{1/2}\nonumber\\
&\quad\quad\quad\quad\quad\quad\quad\quad\leq C \Big\{\|u_0\|_{L^2(0,T; H^{m+1}(\Om))} + \|f\|_{L^2(0,T; H^{-m+1}(\Om))} +\|h\|_{L^2(\Om)}  \Big\}.
\end{align}
This, combined with (\ref{pco311}) and \eqref{pco314}, gives (\ref{co31re1}).
\end{proof}
\begin{remark}
In addition to the assumptions of Theorem \ref{t3.1}, if $\Om$ is  $C^{m,1}$  and $h=0$, then as a consequence of (\ref{co31re1}) we have
\begin{align}\label{co31re2}
\|\na^m w_\va\|_{L^2(\Om_T)}\leq C \va^{1/2}  \|f\|_{L^2(0,T; H^{-m+1}(\Om))}.
\end{align}
This follows from the estimate
\begin{align}\label{pco316}
 \|u_0\|_{L^2(0,T; H^{m+1}(\Om))}\leq C \|f\|_{L^2(0,T; H^{-m+1}(\Om))},
 \end{align}
 which may be proved by time discretization and reducing the estimate to the well-known $H^{m+1}$ estimate for $2m$-order elliptic systems with constant coefficients in $C^{m,1}$ domains.
\end{remark}

\subsection{$O(\va)$ error estimate in $L^2(0,T; H^{m-1}(\Om))$}

With preparations in the last sections, we are now ready to prove Theorems 1.1.
Let $\mathcal{L}_\va^*,  \mathcal{L}_0^*$ be the adjoint operators of $\mc{L}_\va$ and $\mc{L}_0$, respectively. Let $\Om$ be a bounded $C^{m,1}$ domain and $F\in L^2(0,T; H^{-m+1}(\Om)).$  Suppose that $v_\va, v_0$ are, respectively, weak solutions to
\begin{equation} \label{deq1}
 \begin{cases}
 -\pa_tv_\va+\mathcal{L}^*_\varepsilon v_\varepsilon =F  &\text{ in } \Omega_T,   \\
 Tr (D^\gamma v_\varepsilon)=0  & \text{ on } \partial\Omega \times(0,T),\, 0\leq|\gamma|\leq m-1,  \\
 v_\va=0 & \text{ on } \Omega \times \{t=T\},
\end{cases}
\end{equation}
and
\begin{equation} \label{dhoeq1}
 \begin{cases}
 -\pa_tv_0+\mathcal{L}^*_0 v_0 =F  &\text{ in } \Omega_T,   \\
 Tr (D^\gamma v_0)=0  & \text{ on } \partial\Omega \times(0,T),\, 0\leq|\gamma|\leq m-1,  \\
 v_0=0 & \text{ on } \Omega \times \{t=T\}.
\end{cases}
\end{equation}
Then it is not difficult to find that $v_\va(x,T-t), v_0(x,T-t)$ are solutions to problem (\ref{eq1}) (with homogeneous Dirichlet boundary data) and problem (\ref{hoeq1}), respectively, with $f(x,t)=F(x,T-t), h=0$ as well as the coefficient matrix $A(x/\va, t/\va^{2m})$ replaced by $A^*(x/\va, (T-t)/\va^{2m})$.
Note that $A^*(y, T-s)$ satisfies the conditions (\ref{cod1}) and (\ref{cod2}) as $A(y,s).$
Similar to Section 2, we can introduce the matrix of correctors $\chi_{T,\va}^{*} $ and flux correctors $\mathfrak{B}_{T,\va}^{*}(y,s)$ and also $ \mathcal{B}_{T, \varepsilon}^{*}$ for the family of parabolic operators $$-\pa_t+(-1)^m \sum_{|\alpha|=|\beta|=m}  D^\alpha \Big\{A^{*\alpha \beta}(x/\va,(T-t)/\va^{2m})D^\beta  \Big\}, \quad\va>0.$$

 \begin{proof}[\bf Proof of Theorem \ref{tcd}]
We now use the duality argument \cite{Suslina2017-N,gsjfa2017}  to prove Theorem \ref{tcd}. For simplicity, we assume that
$$  \|u_0\|_{L^2(0,T; H^{m+1}(\Om))} + \|f\|_{L^2(0,T; H^{-m+1}(\Om))}+\|h\|_{L^2(\Om)}  \leq 1.$$
In view of the definition of $w_\varepsilon$, to prove (\ref{tcdre1}) it is sufficient to prove the following estimates,
 \begin{align}
   &\varepsilon^m\big\|\sum_{|\gamma|=m}\chi^\gamma_\va K_\varepsilon(D^\gamma u_0 )\big\|_{L^2(0,T; H_0^{m-1}(\Om))}\leq C \va, \label{ptcd11}\\
   &\Big\|\sum_{\substack{|\beta|=|\gamma|=m\\ \zeta+\eta =\gamma\\0\leq|\zeta|\leq m-1}}\varepsilon^{m+|\eta|}(D^\zeta\mathfrak{B}^{\gamma (d+1) \beta})_\varepsilon D^\eta K_\varepsilon(D^\beta u_0)\Big\|_{L^2(0,T; H_0^{m-1}(\Om))}\leq C \va,\label{ptcd12}\\
&\varepsilon^{2m}\Big\|\sum_{|\alpha|=|\beta|=m}\mathcal{B}_\varepsilon^{\alpha\beta}D^\alpha K_\varepsilon(D^\beta u_0)\Big\|_{L^2(0,T; H_0^{m-1}(\Om))}\leq C \va,\label{ptcd13}\\
&\big\|w_\varepsilon\big\|_{L^2(0,T; H_0^{m-1}(\Om))}  \leq C \varepsilon .\label{ptcd1_w}
\end{align}
Thanks to the Poincar\'{e} inequality and Lemmas \ref{le2.2}, \ref{le2.3}, we deduce that
\begin{align}\label{ptcd2}
 &\big\|\sum_{|\gamma|=m}\chi^\gamma_\va K_\varepsilon(D^\gamma u_0  )\big\|_{L^2(0,T; H_0^{m-1}(\Om))} \nonumber\\
&\leq C\sum_{\substack{|\al|=m-1\\ \eta+\zeta'+\zeta=\al}} \big\|  \sum_{|\gamma|=m}  D^\eta \chi^\gamma_\va \, D^{\zeta} S^2_\varepsilon(D^\gamma u_0  ) \, D^{\zeta'}\rho_\varepsilon\varrho_\va  \big\|_{L^2(\Om_T)} \nonumber \\
&\leq C \va^{-m+1} \|\na^m u_0\|_{L^2(\Om_T)} \leq C\va^{-m+1},
\end{align}
which implies (\ref{ptcd11}). In a similar way, we can also get (\ref{ptcd12}) and \eqref{ptcd13}. To prove (\ref{ptcd1_w}), it suffices to verify
\begin{align}\label{ptcd14}
\Big|\int_0^T \langle F, w_\va\rangle_{H^{-m+1}(\Om) \times H_0^{m-1}(\Om)}dt \Big|\leq C \va\|F\|_{L^2(0,T; H^{-m+1}(\Om))}, \quad \forall\, F\in L^2(0,T; H^{-m+1}(\Om)) .
\end{align}

Similar to \eqref{def_w}, we define
\begin{align}\label{ww}
\Phi_\va(x,t)&=v_\varepsilon(x,T-t)-v_0(x, T-t)-\varepsilon^m\sum_{|\gamma|=m}\chi^{*\gamma}_{T,\va} \widetilde{K}_\varepsilon(D^\gamma v_0(x, T-t))
\nonumber\\  &\quad +(-1)^{m+1} \sum_{\substack{|\beta|=|\gamma|=m\\ \zeta+\eta =\gamma\\0\leq|\zeta|\leq m-1}}\varepsilon^{m+|\eta|}(D^\zeta\mathfrak{B}^{*\gamma (d+1) \beta}_T)_\varepsilon D^\eta\widetilde{K}_\varepsilon(D^\beta v_0(x, T-t)) \nonumber\\
  &\quad +(-1)^m\varepsilon^{2m}\sum_{|\alpha|=|\beta|=m}\mathcal{B}_{T, \varepsilon}^{*\alpha\beta}D^\alpha \widetilde{K}_\varepsilon(D^\beta v_0(x, T-t)),
\end{align} where $\widetilde{K}_\varepsilon(u)=S^2_\varepsilon(u)\widetilde{\rho}_\varepsilon\wt{\varrho}_\varepsilon$, $\widetilde{\rho}_\varepsilon\in C_c^\infty(\Omega)$ and $\wt{\varrho}_\varepsilon\in C_c^\infty(0,T)$ such that
\begin{align*}
\begin{split}
&supp(\widetilde{\rho}_\varepsilon)\subset \Om\!\setminus\!\Om_{12\varepsilon} ,\quad  supp(\widetilde{\varrho}_\varepsilon) \subset (12\va^{2m}, T-12\va^{2m}), \\
& 0\leq\widetilde{\rho}_\varepsilon\leq 1,\quad \wt{\rho}_\varepsilon=1\text{ in } \Om\!\setminus\!\Om_{16\varepsilon}
 \quad\text{and} \quad |\nabla^k \widetilde{\rho}_\varepsilon|\leq C \varepsilon^{-k},  1\le k\le m,\\
& 0\leq\widetilde{\varrho}_\varepsilon\leq 1, \quad
 \widetilde{\varrho}_\varepsilon=1  \text{ in } (16\va^{2m}, T-16\va^{2m}) \quad\text{and}\quad
 |  \widetilde{\varrho}^{\,\prime}_\varepsilon|\le C \varepsilon^{-2m} .
\end{split}\end{align*}
Using (\ref{deq1})  and  (\ref{ww}), we deduce that
\begin{align}\label{ptcd15}
&\int_0^T \big\langle F, w_\va\big\rangle_{H^{-m+1}(\Om) \times H_0^{m-1}(\Om)}dt\nonumber\\
&= \int_0^T \big\langle \pa_tw_\va, v_\va(t) \big\rangle_{H^{-m+1}(\Om) \times H_0^{m-1}(\Om)}dt+ \sum_{|\al|=|\be|=m}\int_{\Om_T} A_\va^{\be\al} D^\al w_\va D^\be v_\va(x,t)\nonumber\\
&=\int_0^T \big\langle \pa_tw_\va, \Phi_\va(T-t) \big\rangle_{H^{-m+1}(\Om) \times H_0^{m-1}(\Om)}dt+ \sum_{|\al|=|\be|=m}\int_{\Om_T} A_\va^{\be\al} D^\al w_\va D^\be \Phi_\va(x,T-t)\nonumber\\
&\quad+ \int_0^T \big\langle \pa_tw_\va, v_0(t) \big\rangle_{H^{-m+1}(\Om) \times H_0^{m-1}(\Om)}dt+ \sum_{|\al|=|\be|=m}\int_{\Om_T} A_\va^{\be\al} D^\al w_\va D^\be v_0(x,t)\nonumber\\
&\quad+ \int_0^T \big\langle \pa_tw_\va, v_\va(t)-v_0(t)- \Phi_\va(T-t)\big\rangle_{H^{-m+1}(\Om) \times H_0^{m-1}(\Om)}dt\nonumber\\
& \quad+ \sum_{|\al|=|\be|=m}\int_{\Om_T} A_\va^{\be\al} D^\al w_\va D^\be  \Big\{v_\va(x,t)-v_0(x,t)- \Phi_\va(x,T-t)\Big\}\nonumber\\
&\doteq \mathcal{I}_1 +\mathcal{I}_2+\cdot\cdot\cdot+\mathcal{I}_6.
\end{align}  By Lemma \ref{le3.2}, we have
\begin{align}\label{ptcd160}
\mathcal{I}_1 +\mathcal{I}_2 &\leq   C\left\{\varepsilon \|\nabla^m \Phi_\va\|_{L^2(\Omega_T)}  +\varepsilon^{1/2}  \|\nabla^m \Phi_\va\|_{L^2(\Omega_{T,4\varepsilon})} \right \}  \times \Big\{ \|\pa_t  u_0\|_{L^2(0,T; H^{-m+1}(\Om))}\nonumber\\ &\quad\quad\quad +\|u_0\|_{L^2(0,T; H^{m+1}(\Om))} + \sup_{10\va^{2m}< t< T}\Big(\frac{1}{\va}\int^t_{t-10\va^{2m}} \| \nabla^m  u_0(t)\|^2_{L^2(\Om)} \Big)^{1/2} \Big\} \nonumber\\
  &\leq C  \varepsilon^{1/2} \|\nabla^m \Phi_\va\|_{L^2(\Omega_T)},
\end{align}
where we have used (\ref{pco314}) and (\ref{pco315}) for the second inequality.
Since $A^*(y, T-s)$ satisfies conditions (\ref{cod1}) and (\ref{cod2}) as $A(y,s),$ and $\Om$ is $C^{m,1}$. By (\ref{co31re2}), we have
\begin{align} \label{co31re2'}
 \|\nabla^m \Phi_\va\|_{L^2(\Omega_T)}\leq C \va^{1/2}\|F\|_{L^2(0,T; H^{-m+1}(\Om))},
\end{align}
which combined with (\ref{ptcd160}), implies that
\begin{align}\label{ptcd16}
 \mathcal{I}_1 +\mathcal{I}_2\leq C \va \|F\|_{L^2(0,T; H^{-m+1}(\Om))}.
\end{align}

Likewise, using Lemma \ref{le3.2}, (\ref{pco314}) and (\ref{pco315}),
we deduce that
\begin{align}\label{ptcd17}
\mathcal{I}_3+\mathcal{I}_4& \leq C\left\{\varepsilon \|\nabla^m v_0\|_{L^2(\Omega_T)}  +\varepsilon^{1/2}  \|\nabla^m v_0\|_{L^2(\Omega_{T,4\varepsilon})} \right \}
  \leq C \varepsilon \|F\|_{L^2(0,T; H^{-m+1}(\Om))},
\end{align}
where, for the last step, we have used estimates (\ref{ple3201}), (\ref{pco315}) and (\ref{pco316}) for $v_0$. The proof of these estimates are completely the same as those for $u_0$, since $A^*(y, T-s)$ satisfies conditions as $A(y,s).$

Finally, note that $v_\va(t)-v_0(t)- \Phi_\va(T-t) $ is supported on $\Om_T\!\setminus\!\Om_{T,12\va}$.
From Lemma \ref{le3.2}, (\ref{pco314}) and (\ref{pco315}), it follows that
\begin{align}\label{ptcd18}
\mathcal{I}_5+\mathcal{I}_6&\leq C  \varepsilon\|\nabla^m \big[v_\va-v_0- \Phi_\va(T-t)\big]\|_{L^2(\Omega_T)}
  \times \Big\{\|\pa_t  u_0\|_{L^2(0,T; H^{-m+1}(\Om))} \nonumber\\
   &\quad +\|u_0\|_{L^2(0,T; H^{m+1}(\Om))} + \sup_{10\va^{2m}< t< T}\Big(\frac{1}{\va}\int^t_{t-10\va^{2m}} \| \nabla^m  u_0(t)\|^2_{L^2(\Om)} \Big)^{1/2}\Big\} \nonumber\\
  &\leq C   \varepsilon  \|\nabla^m \big[v_\va-v_0- \Phi_\va(T-t)\big]\|_{L^2(\Omega_T)}.
\end{align}
Similar to (\ref{ptcd2}),  using Poincar\'{e}'s inequality and Lemmas \ref{le2.2}, \ref{le2.3},
we can prove that
\begin{align*}
& \va^{m}\Big\|\sum_{|\gamma|=m}\chi^{*\gamma}_{T,\va} \widetilde{K}_\varepsilon(D^\gamma v_0)\Big\|_{L^2(0,T; H_0^{m}(\Om))}
  \leq C \|v_0\|_{L^2(0,T; H^{m+1}(\Om))},\nonumber\\
&\Big\|\sum_{\substack{|\beta|=|\gamma|=m\\ \zeta+\eta =\gamma\\0\leq|\zeta|\leq m-1}}\varepsilon^{m+|\eta|}(D^\zeta\mathfrak{B}^{*\gamma (d+1) \beta}_T)_\varepsilon D^\eta\widetilde{K}_\varepsilon(D^\beta v_0)\Big\|_{L^2(0,T; H_0^{m}(\Om))} \leq C \|v_0\|_{L^2(0,T; H^{m+1}(\Om))},\nonumber\\
  &\varepsilon^{2m}\Big\|\sum_{|\alpha|=|\beta|=m}\mathcal{B}_{T, \varepsilon}^{*\alpha\beta}D^\alpha \widetilde{K}_\varepsilon(D^\beta v_0)\Big\|_{L^2(0,T; H_0^{m}(\Om))} \leq C \|v_0\|_{L^2(0,T; H^{m+1}(\Om))},
\end{align*}
which, together with (\ref{ww}), (\ref{ptcd18}) and (\ref{pco316}) (for $v_0$), gives
\begin{align} \label{ptcd19}
 \mathcal{I}_5+\mathcal{I}_6&\leq C \va \|F\|_{L^2(0,T; H^{-m+1}(\Om))}.
\end{align}
Note that (\ref{ptcd14}) follows directly from (\ref{ptcd15}), (\ref{ptcd16}), (\ref{ptcd17}) and (\ref{ptcd19}). The proof is thus  complete.
\end{proof}

\section{Convergence rates for the initial-Neumann problem }
In this section, we provide a concise discussion on the convergence rate in the homogenization of initial-Neumann problem with homogeneous boundary data.
\begin{lemma}\label{le5.1}
  Let $\Omega$ be a bounded Lipschitz domain in $\mathbb{R}^d$,  $A$ satisfy conditions (\ref{cod1}) and (\ref{cod2}). Let $u_\varepsilon\in L^2(0,T; H^m(\Om)), u_0\in L^2(0,T; H^{m+1}(\Om))$ be weak solutions to the initial-Neumann problems (\ref{eq1})  and (\ref{hoeq2}), respectively. Let $w_\va$ be defined as in \eqref{def_w}. Then for any  $\phi\in L^2(0, T; H^m(\Omega))$,
\begin{equation}\label{le51re1}
\int_0^T \langle \pa_t w_\va, \phi\rangle_{\wt{H}^{-m}(\Om)\times H^m(\Om)} dt+\sum_{|\alpha|=|\beta|=m}  \int_{\Om_T} D^\alpha \phi
A_\va^{\alpha\beta}   D^\beta w_\varepsilon =  ( \textbf{\text{the r.h.s  of}  (\ref{le31re1})}).
\end{equation}
\end{lemma}
\begin{proof}
The proof is parallel to that of  Lemma \ref{le3.1} with slight adaptations to Neumann problem.
\end{proof}

\begin{lemma}\label{le5.2}
 Under the assumption of Lemma \ref{le5.1}, we have  for any $\phi \in L^2(0,T; H^m(\Omega)),$
 \begin{align}\label{le52re2}
&\int_0^T \big\langle \pa_t w_\va, \phi\big\rangle_{\wt{H}^{-m}(\Om)\times H^m(\Om)} dt +  \sum_{|\alpha|=|\beta|=m} \int_{\Om_T} D^\alpha \phi   A_\va^{\alpha\beta}
 D^\beta w_{\varepsilon }dxdt \nonumber\\
 &\leq C \Big\{\|u_0\|_{L^2(0,T; H^{m+1}(\Om))}+ \|\pa_t  u_0\|_{L^2(0,T; \wt{H}^{-m+1}(\Om))}+ \sup_{10\va^{2m}< t< T}\Big(\frac{1}{\va}\int^t_{t-10\va^{2m}} \| \nabla^m  u_0(t)\|^2_{L^2(\Om)} \Big)^{1/2} \Big\} \nonumber\\ &\quad  \times \left\{\varepsilon \|\nabla^m \phi\|_{L^2(\Omega_T)}  +\varepsilon^{1/2}  \|\nabla^m \phi\|_{L^2(\Omega_{T,8\varepsilon})} \right \},
\end{align}
where $C$ depends only on $d, n, m, \mu, T$ and $\Omega.$
\end{lemma}
\begin{proof}
The proof is almost the same as that of Lemma \ref{le3.2}, let us omit the details.
\end{proof}

With Lemmas \ref{le5.1} and \ref{le5.2} as preparations, we obtain the following theorem parallel to Theorem \ref{t3.1}. Since the proof is almost the same to the one of Theorem \ref{t3.1}, we omit the details for brevity.
\begin{theorem}\label{th5.1}
Let $\Omega$ be a bounded Lipschitz domain in $\mathbb{R}^d$, and $A$ satisfy conditions (\ref{cod1}) and (\ref{cod2}). Let $u_\varepsilon, u_0\in L^2(0,T; H^m(\Om))$ be weak solutions to initial problems (\ref{eq1}) and (\ref{hoeq2}) with homogeneous Neumann boundary data respectively. Moreover, assume that $u_0\in L^2(0,T; H^{m+1}(\Om))$ and $w_\va$ is defined as \eqref{def_w}.
Then \begin{align}\label{co51re1}
\|\na^m w_\va\|_{L^2(\Om_T)}\leq C \va^{1/2} \Big\{\|u_0\|_{L^2(0,T; H^{m+1}(\Om))}+ \|f\|_{L^2(0,T; \wt{H}^{-m+1}(\Om))} +\|h\|_{L^{2}(\Om)}  \Big\},
\end{align}where $C$ depends only on $d, n, m, \mu, T$ and $\Omega.$
If  in addition  $\Om$ is $C^{m,1}$  and $h=0$, then
\begin{align}\label{co51re2}
\|\na^m w_\va\|_{L^2(\Om_T)}\leq C \va^{1/2}  \|f\|_{L^2(0,T; \wt{H}^{-m+1}(\Om))},
\end{align}where $C$ depends only on $d, n, m, \mu, T$ and $\Omega.$
\end{theorem}

\begin{proof}[\bf Proof of Theorem \ref{tcn}] The proof is completely parallel to that of Theorem \ref{tcd}. Indeed, let $F\in L^2(0,T; H^{-m+1}(\Om)).$  We consider the following initial-Neumann problems
\begin{equation} \label{deq2}
 \begin{cases}
 -\pa_tv_\va+\mathcal{L}^*_\varepsilon v_\varepsilon =F  &\text{ in } \Omega_T,   \\
 N_{m-1-j} (v_\va)=0 & \text{ on } \partial\Omega \times(0,T), \quad  j=0,1,..., m-1,  \\
 v_\va=0 & \text{ on } \Omega \times \{t=T\},
\end{cases}
\end{equation}
and
\begin{equation} \label{dhoeq2}
 \begin{cases}
 -\pa_tv_0+\mathcal{L}^*_0 v_0 =F  &\text{ in } \Omega_T,   \\
 N_{m-1-j} (v_0)=0 & \text{ on } \partial\Omega \times(0,T), \quad  j=0,1,..., m-1,  \\
 v_0=0 & \text{ on } \Omega \times \{t=T\},
\end{cases}
\end{equation}
It is obvious that $v_\va(x,T-t), v_0(x,T-t)$ are solutions respectively to (\ref{eq1}) and (\ref{hoeq2}) with homogeneous Neumann boundary data, and also with $f(x,t)=F(x,T-t), h=0 $ and $A(x/\va, t/\va^{2m})$ replaced by $A^*(x/\va, (T-t)/\va^{2m})$.  Moreover,  $v_0$ still satisfy estimates (\ref{ple3201}), (\ref{pco315}) and (\ref{pco316}).
Define $\Phi_\va$ as (\ref{ww}). Observe that
\begin{align}
&\int_0^T \big\langle F, w_\va\big\rangle_{\wt{H}^{-m+1}(\Om) \times H^{m-1}(\Om)}dt\nonumber\\
&= \int_0^T \big\langle \pa_tw_\va, v_\va(t) \big\rangle_{\wt{H}^{-m+1}(\Om) \times H^{m-1}(\Om)}dt+ \sum_{|\al|=|\be|=m}\int_{\Om_T} A_\va^{\al\be} D^\be w_\va D^\al v_\va.\nonumber
\end{align}
With Lemmas \ref{le5.1}, \ref{le5.2} and Theorem \ref{th5.1} at our disposal, we can perform the same analysis as we did for Theorem \ref{tcd} to derive (\ref{tcnre1}).
\end{proof}

\noindent\textbf{Acknowledgments.}   The authors are much obliged to Professor Zhongwei Shen for the guidance and enlightening discussions. They would also like to express deep gratitude to Professor Russell Brown  and the Department of Mathematics of University of Kentucky for the support and the warm hospitality during the authors' visit.
\bibliographystyle{amsplain}
\bibliography{7}

\vspace{0.2cm}
\noindent Weisheng Niu \\
School of Mathematical Science, Anhui University,
Hefei, 230601, P. R. China\\
E-mail:weisheng.niu@gmail.com\\
%

\noindent Yao Xu \\
Department of Mathematics, Nanjing University,
Nanjing, 210093, P. R. China\\
E-mail:dg1421012@smail.nju.edu.cn\\

 \end{document}